\documentclass[12pt]{amsart}
\usepackage{a4wide}
\usepackage{amssymb}
\usepackage{amsthm}
\usepackage{amsmath}
\usepackage{amscd}
\usepackage{verbatim}
\usepackage[all]{xy}

\numberwithin{equation}{section}

\theoremstyle{plain}
\newtheorem{theorem}{Theorem}[section]
\newtheorem{corollary}[theorem]{Corollary}
\newtheorem{lemma}[theorem]{Lemma}
\newtheorem{proposition}[theorem]{Proposition}

\theoremstyle{definition}

\newtheorem{remark}[theorem]{Remark}
\newtheorem{example}[theorem]{Example}

\theoremstyle{remark}

\newcommand{\OO}{\mathcal O}
\newcommand{\A}{\mathbb{A}}
\newcommand{\R}{\mathbb{R}}

\newcommand{\Q}{\mathbb{Q}}
\newcommand{\Z}{\mathbb{Z}}
\newcommand{\N}{\mathbb{N}}
\newcommand{\C}{\mathbb{C}}

\renewcommand{\H}{\mathbb{H}}
\newcommand{\F}{\mathbb{F}}
\newcommand{\D}{\mathbb{D}}
\newcommand{\X}{\mathbb{X}}
\newcommand{\V}{\mathbb{V}}
\newcommand{\W}{\mathbb{W}}
\newcommand{\LL}{\mathbb{L}}

\newcommand{\zxz}[4]{\begin{pmatrix} #1 & #2 \\ #3 & #4 \end{pmatrix}}

\newcommand{\kzxz}[4]{\left(\begin{smallmatrix} #1 & #2 \\ #3 & #4\end{smallmatrix}\right) }
\newcommand{\kabcd}{\kzxz{a}{b}{c}{d}}

\newcommand{\calC}{\mathcal{C}}

\newcommand{\calE}{\mathcal{E}}

\newcommand{\calH}{\mathcal{H}}

\newcommand{\calM}{\mathcal{M}}

\newcommand{\calO}{\mathcal{O}}
\newcommand{\calP}{\mathcal{P}}

\newcommand{\frakp}{\mathfrak p}

\newcommand{\bs}{\backslash}
\newcommand{\norm}{\operatorname{N}}

\newcommand{\vol}{\operatorname{vol}}
\newcommand{\tr}{\operatorname{tr}}

\newcommand{\sgn}{\operatorname{sgn}}

\newcommand{\Sl}{\operatorname{SL}}

\newcommand{\GSpin}{\operatorname{GSpin}}
\newcommand{\Hasse}{\operatorname{Hasse}}
\newcommand{\CT}{\operatorname{CT}}

\newcommand{\Orth}{\operatorname{O}}

\newcommand{\Char}{\operatorname{char}}
\newcommand{\Mat}{\operatorname{Mat}}
\newcommand{\Spec}{\operatorname{Spec}}

\newcommand{\GL}{\operatorname{GL}}
\newcommand{\SO}{\operatorname{SO}}

\newcommand{\Res}{\operatorname{Res}}

\newcommand{\supp}{\operatorname{supp}}

\newcommand{\dv}{\operatorname{div}}

\newcommand{\ord}{\operatorname{ord}}
\newcommand{\kay}{k}
\newcommand{\Gspin}{\operatorname{GSpin}}

\newcommand{\ff}{\hbox{if }}
\newcommand{\SL}{\operatorname{SL}}
\newcommand{\Diff}{\operatorname{Diff}}
\newcommand{\B}{\mathbb B}

\begin{document}

\title[CM values of automorphic Green functions]{CM values of automorphic Green functions on orthogonal groups over totally real fields}

\author[Jan H.~Bruinier and Tonghai Yang]{Jan
Hendrik Bruinier and Tonghai Yang}
\dedicatory{To Stephen S. Kudla}

\address{Fachbereich Mathematik,
Technische Universit\"at Darmstadt, Schlossgartenstrasse 7, D--64289
Darmstadt, Germany}
\email{bruinier@mathematik.tu-darmstadt.de}
\address{Department of Mathematics, University of Wisconsin Madison, Van Vleck Hall, Madison, WI 53706, USA}
\email{thyang@math.wisc.edu} \subjclass[2000]{14K22; 11G15, 11F55, }

\thanks{The first author is partially supported DFG grant BR-2163/2-1.
The second author is partially supported by grants NSF DMS-0855901
and NSFC-10628103.}


\date{\today}

\begin{abstract}
  Generalizing work of Gross--Zagier and Schofer on singular moduli,
  we study the CM values of regularized theta lifts of harmonic
  Whittaker forms.  We compute the archimedian part of the height
  pairing of arithmetic special divisors and CM cycles on Shimura
  varieties associated to quadratic spaces over an arbitrary
  totally real base field.  As a special case, we obtain an explicit
  formula for the norms of the CM values of those meromorphic modular
  forms arising as regularized theta lifts of holomorphic Whittaker
  forms.
\end{abstract}

\maketitle


\section{Introduction}

The values of the classical $j$-function at complex multiplication points  (CM points) are known as singular moduli. They play an important role in number theory, for instance, they are algebraic integers that generate Hilbert class fields of imaginary quadratic fields, and they parametrize elliptic curves with complex multiplication.
Gross and Zagier found a beautiful explicit formula for the prime factorization of the norm of (the difference of two) singular moduli \cite{GZsingular}. A striking consequence is that the prime factors
are small:
If $E$ is an elliptic curve with complex multiplication by the maximal order of an imaginary quadratic field of discriminant $\Delta<0$,  then
any prime dividing the norm of $j(E)$ must divide
$\frac{1}{4}(-3\Delta-x^2)$ for some integer $x$ with $|x|<\sqrt{3|\Delta|}$ and $x^2\equiv  \Delta\pmod{4}$.


The work of Gross and Zagier has inspired a lot of subsequent research in different directions. For instance, Dorman  relaxed some technical assumptions \cite{Do1}, and obtained an analogue for rank 2 Drinfeld modules \cite{Do2}.
Elkies considered Hauptmodules on certain genus zero compact Shimura curves and computed (partly numerically) some of their CM values \cite{El}, \cite{El2}.
Lauter found a conjecture on the primes occuring in the denominators of the CM values of Igusa's $j$-invariants on the moduli space of principally polarized abelian surfaces. Bounds for the denominators of these invariants have interesting applications for the construction of CM genus 2 curves in cryptography. See \cite{GL1}, \cite{GL2}, \cite{YaGeneral} for results in this context.

The $j$-function is an example of a Borcherds product  for the group $\Sl_2(\Z)$, and the difference $j(z_1)-j(z_2)$ of two $j$-functions  is an example of a Borcherds product for $\Sl_2(\Z)\times \Sl_2(\Z)$. Therefore it is natural to ask  whether similar factorization formulas can be obtained for such modular forms in greater generality. An affirmative answer was given by the authors of the present paper  for the values of Borcherds products on Hilbert modular surfaces at `big' CM cycles \cite{BY1}, and by Schofer for the values of Borcherds products at `small' CM cycles \cite{Scho}.

The argument of Schofer is very natural and is based on a method introduced in \cite{Ku:Integrals}. It employs the construction of Borcherds products as regularized theta lifts \cite{Bo2}, the idea of see-saw dual pairs \cite{KuSeesaw}, and the Siegel--Weil formula \cite{We2}, \cite{KR1}.
As an application,  Errthum derived factorization formulas for the norms of CM values of Hauptmodules on Shimura curves of genus zero associated to quaternion algebras over $\Q$ \cite{Err}. In particular, he verified the numerical values computed by Elkies  for such Shimura curves.
The authors of the present paper extended Schofer's approach  to compute the CM values of automorphic Green functions \cite{Br1}, \cite{BF}, yielding a direct link between certain height pairings and central derivatives of Rankin-Selberg $L$-functions, see \cite{BY}. As an application, they obtained a new proof of the Gross-Zagier formula for the canonical heights of Heegner points on modular curves \cite{GZ}.

Borcherds products are certain meromorphic modular forms for the orthogonal group of a quadratic space {\em over $\Q$} of signature $(n,2)$, which are constructed as regularized theta lifts of weakly holomorphic elliptic modular forms of weight $1-n/2$.
%
It was already pointed out by Borcherds in \cite{Bo2} that a direct analogue of his construction for totally real number fields other than $\Q$ cannot exist. This would be a lifting from weakly holomorphic Hilbert modular forms of typically negative weight  to meromorphic modular forms for the orthogonal group of a quadratic space over a totally real number field  $F$. But according to the Koecher principle, any weakly holomorphic Hilbert modular form is automatically holomorphic at the cusps as well.

As a solution to this problem the first author proposed in \cite{Br2} to use harmonic  ``Whittaker forms''  (see Section \ref{sect:3.2}) as input data for a regularized theta lift over an arbitrary totally real field $F$. He constructed a map which  produces meromorphic modular forms and automorphic Green functions on orthogonal groups and which reduces to the Borcherds lift in the special case where the ground field is $\Q$. For instance, for a Shimura curve associated to a quaternion algebra over a totally real field, one obtains meromorphic modular forms whose whose zeros and poles lie on CM points.

In the present paper we find an explicit formula for the CM values of regularized theta lifts of harmonic Whittaker forms, thereby extending the results of \cite{Scho} and the results of \cite{BY} on archimedian height pairings to quadratic spaces over arbitrary totally real base fields $F$. Although the basic idea is the same as in Schofer's work, several new aspects arise. For instance, the regularized integral is {\em not} defined as an integral over a (truncated) fundamental domain for the Hilbert modular group, but as an integral over a fundamental domain for the subgroup of translations.
Moreover, since the Shimura variety we work with is defined over $F$, the archimedian height pairing consists of several local contributions. For the individual pieces we obtain an integral representation (Theorem \ref{thm:fund}), which cannot be evaluated in finite form. Only if we piece together the local heights at all archimedian places, we obtain a finite expression, which yields (for weakly holomorphic input)  the prime factorization of the norm of a CM value (Theorem \ref{thm:fundtr}).
As an example we compute the CM values of the Hauptmodule on a Shimura curve considered by Elkies associated to a quaternion algebra over the maximal totally real field subfield of $\Q(\zeta_7)$.

We now describe the main results of this paper in more detail.
Let $F$ be
a totally real number field of degree $d$ and discriminant $D$.
Let $\sigma_1,\dots,\sigma_d$ be the different embeddings of $F$ into
$\R$.  We write
$\calO_F$ for
the ring of integers and $\partial_F$ for the different of $F$. Let
$(V,Q)$ be a quadratic space over $F$ of dimension $\ell=n+2$. We
assume that  $V$ has signature $(n,2)$ at the place $\sigma_1$ of $F$,
and that $V$ is positive definite at all other archimedian places.
Throughout we assume that $V$ is anisotropic over $F$ or has Witt rank
over $F$ less than $n$. By the assumption on the signature this is
always the case if $d>1$ or $n>2$.

We consider the algebraic group $H=\Res_{F/\Q}\GSpin(V)$ over $\Q$
given by Weil restriction of scalars.
We realize the hermitian symmetric space associated to $H(\R)$ as the Grassmannian $\D$ of oriented negative $2$-planes in $V\otimes_{F,\sigma_1}\R$.
For a compact open subgroup $K\subset H(\hat \Q)$  we
consider the Shimura variety
\begin{align*}
X_{K,1}=H(\Q)\bs (\D\times H(\hat \Q))/K.
\end{align*}
It is a complex quasi-projective variety of dimension $n$, which has a
canonical model over $\sigma_1(F)$, see \cite{Shih}. To simplify the exposition
we assume throughout this introduction that $K$ stabilizes an even unimodular $\calO_F$-lattice $L\subset V$, that $n$ is even, and that
the Shimura variety $X_{K,1}$ is connected. These assumptions are not required in the body of the paper.

It is a special feature of such Shimura varieties that they come with a large supply of algebraic cycles given by quadratic subspaces of $V$, see \cite{Ku:Duke}.
For instance, let $W\subset V$ be a totally positive definite subspace of dimension $n$ defined over $F$. Then the orthogonal complement $W^\perp$ is definite of dimension $2$,
and
we obtain two points $z_W^\pm\in \D$ given by $W^\perp\otimes_{F,\sigma_1}\R$ with the two possible choices of an orientation. Let  $H_W$ be the pointwise
stabilizer of $W$ in $H$,  so that $H_W\cong \Res_{F/\Q}\GSpin(W^\perp)$.
The natural map
\begin{align*}
H_W(\Q)\bs \{z_W^{\pm}\}\times H_W(\hat\Q)/(H_W(\hat \Q)\cap K)
\longrightarrow X_{K,1}
\end{align*}
defines a dimension $0$-cycle $Z_1(W)$ on $X_{K,1}$, which is rational
over $\sigma_1(F)$, see \cite{Ku:Duke}. Since
$\GSpin(W^\perp)$ can be identified with $k^\times$
for a totally imaginary quadratic extension $k$ of $F$, the cycle
$Z_1(W)$ is called the CM cycle associated to $W$.


A principal part polynomial is a Fourier polynomial of the the Form
\begin{align*}
\calP= \sum_{\substack{m\in \partial_F^{-1}\\ m\gg 0}}  c(m) q^{-m},
\end{align*}
where $q^{m}=e^{2\pi i \tr(m\tau)}$ for $\tau\in \H^d$. We let
$Z_1(\calP)=\sum_{\substack{ m\gg 0}}  c(m) Z_1(m)$
be the corresponding linear combination of special divisors $Z_1(m)$
on $X_{K,1}$, see Sections \ref{sect:2.1} and \ref{sect:7.1}.
The principal part polynomial is called weakly holomorphic, if
\begin{align*}
\sum_{\substack{ m\gg 0}}  c(m) b(m)=0,
\end{align*}
for all Hilbert cusp forms $g=\sum_m b(m)q^m$ of parallel weight $1+n/2$ for $\Sl_2(\calO_F)$.

It is proved in \cite{Br2}, that if $\calP$ is a weakly holomorphic principal part polynomial with integral coefficients $c(m)$, then there exists a meromorphic modular form
$R_\calP(z,h)$ on $X_{K,1}$ with divisor $Z_1(\calP)$, which is defined over $F$,
and whose logarithm is essentially given by a regularized theta lift of a weakly holomorphic Whittaker form corresponding to $\calP$, see also Theorem \ref{thm:bop}.
The weight of $R_\calP$ is given  by the degree of the divisor $Z_1(\calP)$, or equivalently by the coefficients of a certain Hilbert Eisenstein series of weight $1+n/2$.
%
In this introduction we assume for simplicity that $\deg(Z_1(\calP))=0$, so that $R_\calP$ has weight zero.
One of our main results (see Corollaries \ref{cor:s1} and \ref{cor:s2}) describes the norm of the CM value
\[
R_\calP(Z_1(W)) = \prod_{[z,h]\in Z_1(W)} R_\calP(z,h)\in \sigma_1(F).
\]

\begin{theorem}
\label{thm:intro1}
Let $\calP$ and $W$ be as above and assume that  $Z_1(W)$ and $Z_1(\calP)$ do not intersect on $X_{K,1}$.
Then the norm of the CM value $R_\calP(Z_1(W))$ is given by
\begin{align*}
\log|\norm_{\sigma_1(F)/\Q} R_\calP(Z_1(W))|
&=\deg(Z_1(W))\left(\log(C)-\frac{1}{4} \CT\langle\calP,\Theta_P\otimes\calE_N^{(0)}\rangle\right).
\end{align*}
Here $\Theta_P$ is the (vector valued) theta function of the totally positive definite lattice $P=W\cap L$, and $\calE_N^{(0)}$ is the `holomorphic part' of a parallel weight $1$ incoherent Hilbert Eisenstein series associated to the lattice $N=W^\perp \cap L$, see Section \ref{sect:Eisenstein}. Moreover, $\CT(\cdot)$ denotes the constant term of a $q$-series, and
$C\in \R_{>0}$ is a normalizing constant which only depends on $\calP$ (and not on  $W$).
\end{theorem}

The constant $C$ can often be determined by specifying the value of $R_\calP$ at some special CM point (see Section \ref{sect:8.3}).
As a consequence, we find that the prime factors of the norm of the CM value are small (see Corollary \ref{cor:s1}), in analogy to the situation for the classical $j$-function.

\begin{corollary}
\label{cor:intro1}
Let $S(N)$
be the set of finite primes $\mathfrak p$ of $F$ for which  $N_{\mathfrak p}$ is not
  unimodular, and let $S(\calP)$ be the set of totally positive $m\in \partial_F^{-1}$ such that
  $c(m)\ne 0$.
We have
\begin{align*}
\log|\norm_{\sigma_1(F)/\Q} R_\calP(Z_1(W))| &=\deg(Z_1(W))\log(C)+
\sum_{\text{$p$ prime}} \alpha_p \log(p)
\end{align*}
with coefficients $\alpha_p\in \Q$, and $\alpha_p=0$ unless  there is a prime $\mathfrak p$  of $F$ above $p$  which belongs to $S(N)$ or
$\mathfrak p| (m-Q(\nu))\partial_F$ for some $m \in S(\calP)$ and $\nu \in P'$ with $m -Q(\nu)\gg 0$ (totally positive).
In particular, $\alpha_p=0$ unless $p \le \max( M(\calP),  |N'/N|, D)$, where
$$
  M(\calP)=\max\{\norm(m)D;\; m\in S(\calP)\}.
$$
\end{corollary}

In Section \ref{sect:8}, we consider the Shimura curve $X$ associated to the triangle group $G_{2,3,7}$ as an example.
It is a genus zero curve with a number of striking properties.
For instance, the minimal quotient area of a discrete subgroup of $\operatorname{PSL}_2(\R)$ is $1/42$, and it is only attained by the triangle group $G_{2,3,7}$.
Elkies constructed a generator $t$ of the function field of $X$ and computed its values at certain CM points, see \cite[Section 5.3]{El} and \cite[Section 2.3]{El2}.

Let $F =\mathbb Q(\zeta_7)^+$ be the maximal totally real subfield of the cyclotomic field $\mathbb Q(\zeta_7)$, where $\zeta_7=e^{2\pi i/7}$.
Then $F$ is a cubic Galois extension of $\Q$ which is generated by $\alpha=\zeta_7+\zeta_7^{-1}$.
%
Let $B$ be the (up to isomorphism unique) quaternion algebra over
$F$ which is split at the first infinite and all finite places and ramified at the second and the third infinite place.
Let $\OO_B$ be a fixed maximal order of $B$, and let $\OO_B^1$ be the group of norm $1$ elements.
The Shimura curve $X$ can be described as the quotient $X=\OO_{B}^1 \backslash \mathbb H$. It can also be described as a Shimura variety $X_{K,1}$ associated to the three-dimensional quadratic space over $F$ given by the trace zero elements of $B$ with the reduced norm as the quadratic form.

The elliptic fixed points $P_3$, $P_4$, $P_7$ of orders $3$, $2$, and $7$ are rational CM points of discriminant $-3$, $-4$, and $-7$, respectively. Elkies considered the rational function $t$ on $X$ that has a double zero at $P_4$, a septuple pole at $P_7$, and that takes the value $1$ at $P_3$. Here we show that this function is a regularized theta lift in the sense of Theorem~\ref{thm:bop}. Employing Theorem \ref{thm:intro1}, we verify some of Elkies' computations and determine some further CM values of $t$. The results are summarized in Table~\ref{tab:1} at the end of this paper.
For instance, for the CM point $P_{11}$ of discriminant $-11$, we obtain that
\[
t(P_{11})=\pm
\frac{ 7^{3}\cdot  11\cdot 43^{2}\cdot 127^{2} \cdot 139^{2}\cdot 307^{2}\cdot
  659^{2}}{3^3\cdot 13^{7}\cdot 83^{7}},
\]
agreeing with the computation of Elkies.

Besides CM values of meromorphic modular forms that arise as liftings of weakly holomorphic Whittaker forms, we compute CM values of automorphic Green functions associated to special divisors (see Section \ref{sect:6} and Theorem \ref{thm:fundtr} in Section \ref{sect:7}). In this more general situation the CM value is essentially given by the sum of two terms.
The first term is the right hand side of the formula of Theorem \ref{thm:intro1}. The second term is the central derivative $L'(\xi(\calP), W,0)$ of a Rankin convolution $L$-function of the theta function $\Theta_P$ and a Hilbert cusp form $\xi(\calP)$ of parallel weight $1+n/2$ associated to the principal part polynomial $\calP$.
Similarly as in \cite[Section~5]{BY}, the first term should be the negative of a finite intersection pairing, so that the second term should be the height pairing of an arithmetic special divisor and a CM cycle.
Note that our approach also gives a formula for the integrals of automorphic Green functions analogously to \cite{Ku:Integrals}, \cite{BK}, see Remark \ref{rem:integrals}.

This paper is organized as follows: In Section \ref{sect:2} we set up the notation and define the Shimura variety and its special cycles. In Section \ref{sect:3} we recall from \cite{Br2} some facts on Whittaker forms, and in Section \ref{sect:4} we collect some material on Siegel theta functions, Eisenstein series, and the Siegel--Weil formula. In particular we carefully analyze incoherent Hilbert Eisenstein series. Section \ref{sect:5} deals with the regularized theta lift of Whittaker forms, and Section \ref{sect:6} contains a preliminary result on CM values of automorphic Green functions (Theorem \ref{thm:fund}). In Section \ref{sect:7} we put together the computations at the different archimedian places to obtain our main result, Theorem \ref{thm:fundtr}. It is convenient to formulate it using the concept of incoherent adelic quadratic spaces.
Finally, in Section \ref{sect:8} we consider the example studied in \cite[Section 5.3]{El}.

It is a pleasure to dedicate this paper to Steve Kudla on the occasion of his 60th birthday.
It is clear that this paper is greatly influenced by many beautiful ideas appearing in his work.
We thank him for his constant support. Moreover, we thank the referee, N. Elkies, and  J. Voight for useful comments.

\section{Quadratic spaces and Shimura varieties}
\label{sect:2}

We use the same setup and the same notation as in \cite{Br2}.
Let $F$ be a totally real number field of degree $d$ over $\Q$. We write $\calO_F$ for the ring of integers in $F$, and write $\partial=\partial_F$ for the different ideal of $F$.
The discriminant of $F$ is denoted by $D=\norm(\partial)=\#\calO_F/\partial$.
Let $\sigma_1,\dots,\sigma_d$ be the different embeddings of $F$ into $\R$.
We write $\A_F$ for the ring of adeles of $F$ and $\hat F$ for the subring of finite adeles.

Let $(V,Q)$ be a non-degenerate quadratic space of dimension $\ell=n+2$ over $F$.
We put $V_{\sigma_i}=V\otimes_{F,\sigma_i}\R$ and identify
$V(\R)=V\otimes_\Q\R=\bigoplus_i V_{\sigma_i}$. We assume that $V$ has signature
\[((n,2),(n+2,0),\dots, (n+2,0)),
\]
that is, $V_{\sigma_1}$ has signature $(n,2)$ and $V_{\sigma_i}$ has signature $(n+2,0)$ for $i=2,\dots ,d$.
Throughout we assume that $V$ is anisotropic over $F$ or has Witt rank over $F$ less than $n$. By the assumption on the signature this is always the case if $d>1$ or $n>2$.

Let $\GSpin(V)$ be the `general' Spin group of $V$, that is, the group of all invertible elements $g$ in the even Clifford algebra of $V$ such that $gVg^{-1}=V$.
It is an algebraic group over $F$, and the vector representation gives rise to an exact sequence
\[
1\longrightarrow F^\times\longrightarrow  \GSpin(V)\longrightarrow  \SO(V)\longrightarrow 1.
\]
We consider the algebraic group $H=\Res_{F/\Q}\GSpin(V)$ over $\Q$ given by Weil restriction of scalars. So
for any $\Q$-agebra $R$, we have $H(R)=\GSpin(V)(R\otimes_\Q F)$. In particular,
$H(\Q)$ can be identified with $\GSpin(V)(F)$.

We realize the hermitean symmetric space corresponding to $H$ as  the Grassmannian $\D$ of oriented negative definite $2$-dimensional subspaces of $V_{\sigma_1}$.
Note that $\D$ has two components corresponding to the two possible choices of the orientation.

For $K\subset H(\hat\Q)$ compact open we consider the Shimura variety
\begin{align}
X_K:=H(\Q)\bs (\D\times H(\hat \Q))/K.
\end{align}
It is a complex quasi-projective variety of dimension $n$, which has a canonical model over $\sigma_1(F)$, see \cite{Shih}.
It is projective if and only if $V$ is anisotropic over $F$.


Let $L\subset V$ be an $\calO_F$-lattice, that is, a finitely generated $\calO_F$-submodule such that $L\otimes_{\calO_F} F=V$.
We assume that $L$ is even, that is, $Q(L)\subset \partial^{-1}$.
Then $Q_\Q(x)=\tr_{F/\Q}Q(x)$ defines an even $\Z$-valued quadratic form on $L$.
Let $L'$ be the $\Z$-dual lattice of $L$ with respect to the quadratic form $Q_\Q$. It is also an $\calO_F$-lattice.
The finite $\calO_F$-module $L'/L$ is called the discriminant group of $L$.
The lattice $L$ is called unimodular if $L'=L$.

We write $\hat L= L\otimes_\Z \hat\Z$, where $\hat \Z=\prod_p \Z_p$.
Recall that $H(\hat \Q)$ acts on the set of lattices $M\subset V$ by $M\mapsto h M:= (h\hat M)\cap V(F)$.
This action induces an isomorphism $M'/M\to (h M)'/(h M)$, and $hM$ lies in the same genus as $M$.
Throughout we assume that the compact open subgroup $K\subset H(\hat\Q)$ fixes the lattice $L\subset V$ and acts trivially on $L'/L$.

\subsection{Special cycles}

\label{sect:2.1}

Here we recall the special cycles on $X_K$ which were introduced by Kudla in \cite{Ku:Duke}.
Let $W\subset V$ be a subspace of dimension $r$ which is defined over $F$ and totally positive definite.
We write
$W^\perp$ for the orthogonal complement of $W$ in $V$ and $H_W$ for the pointwise
stabilizer of $W$ in $H$. So $H_W\cong \Res_{F/\Q}\GSpin(W^\perp)$. The
sub-Grassmannian
\begin{align*}
\D_{W}=\{ z\in \D;\;\text{$z\perp W$}\}
\end{align*}
defines an analytic submanifold of $\D$. For $h\in H(\hat\Q)$ we consider
the natural map
\begin{align*}
\label{eqY4.3}
H_W(\Q)\bs \D_W\times H_W(\hat\Q)/(H_W(\hat \Q)\cap hKh^{-1})
\longrightarrow X_K,\quad (z,h_1)\mapsto (z,h_1 h).
\end{align*}
Its image defines a cycle   $Z(W,h)$ of codimension $r$ on $X_K$, which is rational
over $\sigma_1(F)$, see \cite{Ku:Duke}.

In the present paper we are interested in two particular cases of this construction.
First, if $r=n$, then $W$ is a maximal totally positive definite subspace and $W^\perp$ is definite of signature
\[
((0,2),(2,0),\dots,(2,0)).
\]
The Grassmannian $\D_W$ consists of $2$ points $z_W^\pm$, given by $W^\perp\otimes_{F,\sigma_1} \R$ with the two possible choices of the orientation.
The group $\GSpin(W^\perp)$ can be identified with $k^\times$ for a totally imaginary quadratic extension $k$ of $F$. For this reason
the corresponding dimension $0$ cycle
$Z(W)=Z(W,1)$ is called the CM cycle associated to $W$.

Second, if $x\in V$ is a vector of totally positive norm, we may consider the one-dimensional subspace $Fx\subset V$.
For $h\in H(\hat\Q)$ we obtain a divisor $Z(Fx,h)$ on $X_K$. We consider certain sums of these divisors, called weighted divisors.
They generalize Heegner divisors on modular curves.
Let $m\in F$ be totally positive, and let
$\varphi\in S(V(\hat F))^K$ be a $K$-invariant Schwartz function.
If there is an $x_0\in V(F)$ with $Q(x_0)=m$, we define the weighted divisor
\begin{align*}
Z(m,\varphi)= \sum_{h\in H_{x_0}(\hat\Q)\bs  H(\hat\Q)/K } \varphi(h^{-1} x_0) Z(Fx_0,h).
\end{align*}
The sum is finite, and $Z(m,\varphi)$ is a divisor on $X_K$ with complex coefficients.
If there is no $x_0\in V(F)$ with $Q(x_0)=m$, we put $Z(m,\varphi)=0$.
If $\mu\in L'/L$ is a coset, and $\chi_\mu=\Char(\mu+\hat L) \in S(V(\hat F))^K$ is the characteristic function, we briefly write
\[Z(m,\mu):=Z(m,\chi_\mu).
\]

\section{The Weil representation and Whittaker forms}

\label{sect:3}

Let $\H$ be the upper complex half plane.
We use $\tau=(\tau_1,\dots,\tau_d)$ as a standard variable on $\H^d$ and put $u_i=\Re(\tau_i)$, $v_i=\Im(\tau_i)$. For a $d$-tuple $(w_1,\dots,w_d)$ of complex numbers, we put
$\tr(w)= \sum_{i} w_i$ and
$\norm(w)= \prod_{i} w_i$.
We view $\C^d$ as a $\R^d$-module by putting $\lambda w=(\lambda_1 w_1,\dots,\lambda_d w_d)$ for $\lambda=(\lambda_1,\dots,\lambda_d)\in \R^d$.
For $x\in F$ we briefly write $x_i=\sigma_i(x)$ and identify $x$ with its image $(x_1,\dots,x_d)\in \R^d$. The usual trace and norm of $x$ coincide with the above definitions. Moreover, the inclusion $F\to\R^d$ defines an $F$-vector space structure on $\C^d$.

\subsection{The Weil representation}

We are interested in certain  vector valued Hilbert modular forms for the group
$G=\Res_{F/\Q} \Sl_2$. For $g\in  \Sl_2(F)\cong G(\Q)$, we briefly write $g_i=\sigma_i(g)$. So the image of $g$ in $G(\R)\cong \Sl_2(\R)^d$ is given by $(g_1,\dots,g_d)$.
The group $G(\R)$ acts on $\H^d$ by fractional linear transformations.
We denote by $\tilde G_\A$ the twofold metaplectic cover of $G(\A)$. Let  $\tilde G_\R$ be the full inverse image  in $\tilde G_\A$ of $G(\R)$.
We will frequently realize $\tilde G_\R$ as the group of pairs
\[
\left( g, \phi(\tau)\right),
\]
where $g=\kabcd\in G(\R)$ and $\phi(\tau)$ is a holomorphic function on $\H^d$ such that $\phi(\tau)^2=\norm(c\tau+d)$.
The group law is defined in the usual way.
Let $\tilde \Gamma$ be the full inverse image in $\tilde G_\R$ of the Hilbert modular group
\[
\Gamma=\Sl_2(\calO_F)\subset G(\R).
\]
It follows from Vaserstein's theorem \cite{Va} that $\tilde\Gamma$ is generated by the elements
\begin{align*}
T_b &= \left( \zxz{1}{b}{0}{1 }, 1\right), \qquad b\in \calO_F,\\
S &  = \left( \zxz{0}{-1}{1}{0 }, \sqrt{\norm(\tau)}\right),\\
N &  =  \left( \zxz{1}{0}{0}{1 }, -1 \right).
\end{align*}
By slight abuse of notation, we write
\begin{align}
\tilde\Gamma_\infty:=\{T_b\cdot (1,\pm 1);\quad b\in \calO_F\}
\end{align}
for the subgroup of translations of $\tilde \Gamma$. Note that $\Gamma_\infty$ is not the full stabilizer of the cusp $\infty$ when $d>1$.

Let $L\subset V$ be an $\calO_F$-lattice.
For $\mu \in L'/L$ we write $\chi_\mu=\Char(\mu+\hat L)\in S(V(\hat F))$ for the characteristic function of the coset.
Associated to the reductive dual pair $(\Sl_2, \Orth(V))$ there is a Weil representation $\omega=\omega_\psi$ of $\tilde G_\A$ on the Schwartz space $S(V(\A_F))$, where $\psi$ is the standard additive  character of $F\bs \A_F$ with $\psi_\infty (x)= e(\tr x)$ \cite{We1}.
The subspace
\[
S_L= \bigoplus_{\mu\in L'/L}  \C\chi_\mu \subset S(V(\hat F))
\]
is preserved by the action of $\widetilde{\Sl}_2(\hat \calO_F )$, the full inverse image in $\tilde G_\A$ of $\Sl_2(\hat\calO_F)\subset G(\hat \Q)$.
The canonical splitting $G(F)\to  \tilde G_\A$ defines a homomorphism
\[
\tilde \Gamma \longrightarrow   \widetilde{\Sl}_2(\hat \calO_F ), \quad \gamma \mapsto \hat \gamma,
\]
where $\hat \gamma$  is the unique element such that $\gamma \hat \gamma$ is in the image of $G(F)$.
This induces a representation $\rho_L$ of $\tilde \Gamma$ on $S_L$ by
\[
\rho_L(\gamma)\varphi = \bar \omega(\hat \gamma),\quad \gamma\in \tilde\Gamma.
\]
See \cite[Section 3.2]{Br2}, for explicit formulas for the action of $T_b$, $S$, $N$.
We denote the standard $\C$-bilinear pairing on $ S_L$ (the $L^2$ bilinear pairing) by
\begin{align}
\langle a,b\rangle=\sum_{\mu\in L'/L} a_\mu b_\mu
\end{align}
for $a,b \in S_L$.
The representation $\rho_L$ is unitary, that is, we have $\langle \bar\rho_L a,\rho_L b\rangle = \langle a,b\rangle$.

\subsection{Whittaker forms}
\label{sect:3.2}

Here we recall some facts on Whittaker forms, see \cite[Section 4]{Br2} for details.
Let $k=(k_1,\dots,k_d)\in (\frac{1}{2}\Z)^d$ be a weight. Throughout we assume that $k\equiv (\frac{\ell}{2},\dots,\frac{\ell}{2})\pmod{\Z^d}$.
%
We define a Petersson slash operator in weight $k$ for the representation $\rho_L$ on functions $f:\H^d\to S_L$ by
\[
(f\mid_{k,\rho_L}(g,\phi)) (\tau)= (c_1\tau_1+d_1)^{-k_1+\ell/2}\cdots (c_d\tau_d+d_d)^{-k_d+\ell/2} \phi(\tau)^{-\ell} \rho_L(g,\phi)^{-1}f(g\tau),
\]
where $(g,\phi)\in \tilde G_\R$ and $g=\kabcd$.
The Petersson slash operator for the dual representation  $\bar\rho_L$ is defined analogously. We write $S_{k,\rho_L}$ for the space of vector valued Hilbert cusp forms of weight $k$ for $\tilde \Gamma$ with representation $\rho_L$.

We have the usual hyperbolic Laplace operators in weight $k$ acting on smooth functions on $\H^d$. They are given by
\begin{align}
\label{defdelta}
\Delta_k^{(j)} = -v_j^2\left( \frac{\partial^2}{\partial u_j^2}+ \frac{\partial^2}{\partial v_j^2}\right) + ik_j v_j\left( \frac{\partial}{\partial u_j}+i \frac{\partial}{\partial v_j}\right)
\end{align}
for $j=1,\dots,d$.
Moreover, we have the Maass lowering operators
\begin{align*}
L^{(j)}  &= -2i v_j^2 \frac{\partial}{\partial\bar{\tau_j}},
\end{align*}
which lower the weight of an automorphic form in the $j$-th component by $2$.

We recall some properties of Whittaker functions, see  \cite[Chap.~13 pp.~189]{AS}  or \cite[Vol.~I Chap.~6 p.~264]{B1}.
%
Kummer's confluent hypergeometric function is defined by
\begin{align} \label{kummer1}
 M(a,b, z) = \sum_{n=0}^\infty \frac{(a)_n }{(b)_n} \frac{z^n}{n!},
\end{align}
where
$ (a)_n = \Gamma(a+n)/\Gamma(a)$ and $(a)_0=1$.
The Whittaker functions are defined by
\begin{align}
\label{eq:m1}
M_{\nu,\mu}(z)&= e^{-z/2} z^{1/2+\mu} M(1/2+\mu-\nu, 1+2\mu, z),\\
\label{eq:w1}
W_{\nu,\mu}(z)&= \frac{\Gamma(-2\mu)}{\Gamma(1/2-\mu-\nu)}M_{\nu,\mu}(z)
+\frac{\Gamma(2\mu)}{\Gamma(1/2+\mu-\nu)}M_{\nu,-\mu}(z).
\end{align}
They are linearly independent solutions of the Whittaker differential equation.
The $M$-Whittaker function has the asymptotic behavior
\begin{align}\label{Masy1}
M_{\nu,\,\mu}(z) &= z^{\mu+1/2}(1+O(z)), \qquad z\to 0,\\
\label{Masy2}
M_{\nu,\,\mu}(z) &= \frac{\Gamma(1+2\mu)}{\Gamma(\mu-\nu+1/2)}e^{z/2}z^{-\nu} ( 1+O(z^{-1})) ,\qquad z\to \infty,
\end{align}
while $W_{\nu,\,\mu}(z)$ is exponentially decreasing for real $z\to \infty$ and behaves like a constant times $z^{-\mu+1/2}$ as $z\to 0$.
For convenience we put for $s\in\C$ and $v_1\in\R$:
\begin{align}
\label{calM}
\calM_s(v_1)&=
|v_1|^{-k_1/2} M_{ \sgn(v_1)k_1/2,\,s/2}(|v_1|)
\cdot
e^{-v_1/2}.
\end{align}
The special value at
$s_0=1-k_1$
is of particular interest.
If $v_1<0$  we have
\begin{align}
\label{Mspecial}
\calM_{s_0}(v_1)&=  \Gamma(2-k_1)
\left(1-\frac{\Gamma(1-k_1,4\pi |m_1| v_1)}{\Gamma(1-k_1)}\right)
e^{-v_1}
\end{align}
where $\Gamma(a,x)=\int_x^\infty e^{-t}t^{a-1}dt$ denotes the incomplete gamma function.

A {\em Whittaker form} of weight $k$ and parameter $s$ (for $\tilde \Gamma$, $\bar \rho_L$, and $\sigma_1$)
is a finite linear combination of the functions
\begin{align}
\label{eq:deffms}
f_{m,\mu}(\tau,s):=
C(m,k,s)\calM_s(-4\pi m_1 v_1)e(\tr (-m\bar \tau) ) \chi_\mu
\end{align}
for $\mu\in L'/L$, $m\in  \partial^{-1}+Q(\mu)$, and $m\gg 0$. Here $C(m,k,s)$ denotes the normalizing factor
\begin{align}
\label{eq:defcms}
C(m,k,s):=\frac{(4\pi m_2)^{k_2-1}\cdots (4\pi m_d)^{k_d-1}}{\Gamma(s+1)\Gamma(k_2-1)\cdots \Gamma(k_d-1)}.
\end{align}
A {\em harmonic Whittaker form} is a Whittaker form with parameter $s_0=1-k_1$.
We denote  by $H_{k,\bar\rho_L}$ the $\C$-vector space of harmonic Whittaker forms of weight $k$ for $\tilde \Gamma$ and $\bar \rho_L$.

So any $f\in H_{k,\bar\rho_L}$  is a finite linear combination of the functions
\begin{align}
\label{eq:fs0}
f_{m,\mu}(\tau)&:=f_{m,\mu}(\tau,s_0)\\
\nonumber
&\phantom{:}= C(m,k,s_0)\Gamma(2-k_1)
\left(1-\frac{\Gamma(1-k_1,4\pi m_1 v_1)}{\Gamma(1-k_1)}\right)e^{4\pi m_1 v_1}e(\tr(-m\bar\tau))\chi_\mu
\end{align}
for $\mu\in L'/L$, $m\in  \partial^{-1}+Q(\mu)$, and $m\gg 0$.
A harmonic Whittaker form $f$ satisfies $\Delta_k^{(1)} f =0$ and is antiholomorphic  in the  variables $\tau_2, \dots,\tau_d$.

We define the {\em dual weight} for $k$ by $\kappa= (2-k_1,k_2,\dots,k_d)$.
We consider the differential operator $\delta=\delta_k^{(1)}$ on functions $f:\H^d\to  S_L$ given by
\begin{align}
\label{def:delta}
\delta(f)=v_1^{k_1-2} \overline{L^{(1)} f(\tau)}.
\end{align}
If $f\in H_{k,\bar\rho_L}$, then $\delta (f)$ is a holomorphic function satisfying
$f(T_b\tau) =  \rho_L(T_b) f(\tau)$ for all $b\in \calO_F$.
In particular, we have
\begin{align}
\label{eq:delta}
\delta (f_{m,\mu})(\tau)= \frac{(4\pi m_1)^{\kappa_1-1}\cdots (4\pi m_d)^{\kappa_d-1}}{\Gamma(\kappa_1-1)\cdots \Gamma(\kappa_d-1)}
e(\tr(m\tau))\chi_\mu.
\end{align}

For the rest of this section we assume that $\kappa_j\geq 3/2 $ for $j=1,\dots,d$.
We define an operator $\xi=\xi_k^{(1)}$ taking $H_{k,\bar\rho_L}$ to $S_{\kappa,\rho_L}$ by
\begin{align}
\label{def:xi}
\xi (f)=\sum_{\gamma\in \tilde\Gamma_\infty\bs \tilde\Gamma} \delta (f)\mid_{\kappa,\rho_L}\gamma.
\end{align}
When $\kappa_j>2$ for $j=1,\dots,d$, the Poincar\'e series on the right hand side converges normally
and defines a holomorphic cusp form. When $\kappa_j\geq 3/2$ it has to be regularized using ``Hecke summation''.
According to \cite[Proposition 4.3]{Br2}, the map $\xi $ is surjective.

A Whittaker form $f$ is called {\em weakly holomorphic} if it is harmonic and satisfies $\xi (f)=0$.
We denote by $M^!_{k,\bar\rho_L}$ the subspace of weakly holomorphic Whittaker forms in $H_{k,\bar\rho_L}$.
So we have the exact sequence
\begin{align}
\label{ex-sequ}
\xymatrix{ 0\ar[r]& M^!_{k,\bar\rho_L} \ar[r]& H_{k,\bar\rho_L}
\ar[r]^{\xi }& S_{\kappa,\rho_L} \ar[r] & 0 }.
\end{align}

Recall that the Petersson scalar product of $f,g\in S_{\kappa,\rho_L}$ is given by
\begin{align}
(f, g)_{Pet}= \frac{1}{\sqrt{D}}\int_{\tilde \Gamma\bs \H^d} \langle f,\bar g \rangle v^\kappa \,
d\mu(\tau),
\end{align}
where $d\mu (\tau)=\frac{du_1\,dv_1}{v_1^2}\dots \frac{du_d\,dv_d}{v_d^2}$ is the invariant measure on $\H^d$, and $v^\kappa$ is understood in multi-index notation.
We define a bilinear pairing between the spaces $S_{\kappa,\rho_L}$  and $H_{k,\bar\rho_L}$ by putting
\begin{equation}
\label{defpair}
\{g,f\}=\big( g,\, \xi (f)\big)_{Pet}
\end{equation}
for $g\in S_{\kappa,\rho_L}$ and $f\in H_{k,\bar\rho_L}$. The pairing vanishes when $f$ is weakly holomorphic, and the induced pairing between $S_{\kappa,\rho_L}$ and $ H_{k,\bar\rho_L}/M^!_{k,\bar\rho_L}$
is non-degenerate.
The pairing can also be computed using Fourier expansions.
We define the {\em principal part} of
\[
f=\sum_{\mu\in L'/L} \sum_{m\gg 0} c(m,\mu)f_{m,\mu}(\tau)\in H_{k,\bar\rho_L}
\]
to be the $S_L$-valued Fourier polynomial
\begin{align}
\label{eq:prinpart}
\calP(f)=\sum_{\mu\in L'/L} \sum_{m\gg 0} c(m,\mu) q^{-m}\chi_\mu,
\end{align}
where $q^m=e(\tr(m\tau))$. Then for $g\in S_{\kappa,\rho_L}$ with Fourier expansion $g=\sum_{n,\nu}b(n,\nu) q^n\chi_\nu$ we have
\begin{equation}
\label{pairalt}
\{g,f\}= \CT(\left\langle \calP(f) ,g\right\rangle)=\sum_{\mu\in L'/L} \sum_{m\gg 0}  c(m,\mu) b(m,\mu).
\end{equation}
Here $\CT(\cdot)$ denotes the constant term of a holomorphic Fourier series. With this formula it can be checked whether a harmonic Whittaker form is weakly holomorphic.

\section{Siegel theta functions and Eisenstein series}
\label{sect:Eisenstein}
\label{sect:4}

Let $(V, Q)$ be the quadratic space as in Section  \ref{sect:2} and let $L$ be an even $\OO_F$-lattice. Let $\psi_\Q$ be the canonical unramified additive character of $\Q \backslash \A$ such  that $\psi_\infty(x) = e(x)$, and let $\psi=\psi_\Q \circ \tr_{F/\Q}$. Let $\omega=\omega_{V, \psi}$ be the Weil representation of $\tilde G_\A$ on $S(V(\A))$.  In this section, we will review  Siegel theta functions, coherent and incoherent Eisenstein series associated to $L$, and their their relations. We will also study the Fourier expansion of the derivative of the incoherent Eisenstein series. Let $\chi=\chi_V=((-1)^{\frac{\ell(\ell-1)}2} \det V, \cdot )_\A$ be the quadratic Hecke character of $F$ associated to $V$,  where $\det V$ is the determinant of the Gram matrix of
the bilinear form on $V$ with respect to a fixed basis. Let $s_0 =\frac{n}2 =\frac{1}2(\dim  V -2)$.
Throughout we assume that $V$ is anisotropic over $F$ or has Witt rank over $F$ less than $n$.
We put
\begin{align*}
\kappa&=\kappa_V =(\frac{n+2}2, \cdots, \frac{n+2}2), \\
\tilde\kappa &=\tilde\kappa_V=(\frac{n-2}2, \frac{n+2}2, \cdots, \frac{n+2}2).
\end{align*}

\subsection{Eisenstein series}

Let $P=N M$ be the standard Borel subgroup of $G=\SL_2(F)$ with
$$
P=N M =\{ n(b) m(a);\; a \in F^\times, b \in F\}.
$$
Here
$$
J=\zxz {0} {-1} {1} {0}, \quad n(b) =\zxz {1} {b} {0} {1}, \quad
m(a) =\zxz {a} {0} {0} {a^{-1}}.
$$
Let $I(s, \chi) =\hbox{Ind}_{\tilde P_\A}^{\tilde G_\A} \chi |\cdot|^s$ be the induced representation
of $\tilde{G}_\A$ (see for example \cite{KRYbook}, \cite{KR1}, \cite{Ku:Integrals}). A section $\Phi \in I(s, \chi)$ is factorizable if
$\Phi=\otimes \Phi_w$ with local sections $\Phi_w \in I(s, \chi_w)$ at the places $w$ of $F$, and it is standard if its restriction to the standard maximal compact subgroup of $\tilde G_\A$ is independent of $s$.
For a standard factorizable section $\Phi =\otimes \Phi_w\in I(s,
\chi)$, the Eisenstein series
\begin{equation}
E(\tilde g, s, \Phi) =\sum_{\gamma \in P \backslash \SL_2(F)}
\Phi(\gamma \tilde g, s)
\end{equation}
is absolutely convergent for $\Re s \gg 0$, has a meromorphic
continuation to the whole complex planes with finitely many
possible poles, and satisfies a functional equation
\begin{equation}
E(\tilde g, s, \Phi) =E(\tilde g, -s, M(s)\Phi),
\end{equation}
where
$$
M(s)\Phi(g,s) = \int_{\A_F} \Phi(J n(b)g, s) db
$$
is the intertwining operator.
The Eisenstein series has a Fourier expansion
$$
E(\tilde g, s, \Phi) =\sum_{m \in F} E_m(\tilde g, s, \Phi),
$$
where
$$
E_m(\tilde g, s, \Phi) =\prod_{w \le \infty}  W_{m, w}(\tilde g_w, s, \Phi_w)
$$
for $m \ne 0$, and the constant term is given by
$$
E_0(\tilde g, s, \Phi) =\Phi(\tilde g, s) + M(s) \Phi(\tilde g, s) = \Phi(\tilde g, s) +
\prod_{w} W_{0, w}(\tilde g_w, s, \Phi_w).
$$
Here
$$
W_{m, w}(\tilde g_w, s, \Phi_w) =\int_{F_w} \Phi_w(J n(b) \tilde g_w, s)
\psi_w(-m b) db
$$
is the local Whittaker function, and $db$ is the Haar measure on
$F_w$ which is self-dual with respect to $\psi_w$.

When $w=\sigma_j$ is an infinite prime,
we identify $F_{\sigma_j}=F\otimes_{\sigma_j}\R$ with $\mathbb R$. Then  $I(s, \chi_w)$ is generated by
the functions
$$
\Phi_{\mathbb R}^{r_j}([k_\theta, 1]) = e^{i \theta r_j/2 },
\quad -\pi < \theta \le \pi,
$$
with $r_j \equiv\frac{n+2}{2}  \pmod{2}$.
Here
$$
k_\theta=\zxz {\cos \theta} {\sin \theta} {-\sin \theta}  {\cos
\theta},
$$
and  $[g, \epsilon]$ is the normalized coordinate system to write   $\tilde G_w$ as  $G_w \times \{ \pm 1\}$, see  \cite{KRYbook}.
For $k =(k_1, \cdots, k_d) \equiv (\frac{n+2}2, \cdots,
\frac{n+2}2) \pmod{2}$, we write $\Phi_\infty^k = \prod_j
\Phi_{\mathbb R}^{k_j} \in I(s, \chi_\infty)$.

\subsection{The map $\lambda$}
Let $w$ be a place of $F$.
Locally, for a quadratic space $(V_w, Q)$  of dimension $n+2$ with quadratic character $\chi_w$, there is a $\widetilde{\SL}_2(F_w)$-equivariant map
\begin{equation}
\lambda: \,  S(V(F_w)) \rightarrow I(s_0, \chi_w),  \quad
\lambda(\phi)(\tilde g) = \omega(\tilde g) \phi(0).
\end{equation}
We will also write $\lambda(\phi)$ for the  unique standard
section in $I(s, \chi_w)$ whose value at $s=s_0$ is $\lambda(\phi)$.
When $(V, Q)$ is a global quadratic space over $F$, we have a
$\widetilde{\Sl}_2(\A_F)$-equivariant map $$ \lambda: \,   S(V(\A_F))  \rightarrow I(s_0,
\chi),  \quad \lambda(\phi)(\tilde g) = \omega(\tilde g) \phi(0).
$$
The following local facts are well-known.

\begin{lemma} \label{lem:lambdainfty}
Let $V$ be a  quadratic space over $\mathbb R$ of signature $(p, q)$ with a specific orthogonal decomposition
$V =V^+ \oplus V^-$ into a positive definite space $V^+$ and a negative definite space $V^-$.  Let
$$
\phi^{p, q}(x) = e^{-2 \pi Q(x_+) + 2 \pi Q(x_-)} \in S(V),
$$
where $x =x_+ + x_- \in V$ with $x_{\pm} \in V_{\pm}$.
Then
$$
\lambda(\phi^{p, q}) = \Phi_{\mathbb R}^{\frac{p-q}2}.
$$
\end{lemma}

\begin{lemma} \label{lem:lambdaunram}
Let $\mathfrak p$ be a  finite prime of $F$, and assume that $L_{\mathfrak p}$ is a unimodular lattice in $V_{\mathfrak p}$. Let $\chi_0 \in S(V_{\mathfrak p})$ be the characteristic function of $L_{\mathfrak p}$. Then $\Phi_{\mathfrak p}^0=\lambda(\chi_0)$ is the spherical section in  $I(s_0, \chi_{\mathfrak p})$ with $\lambda(\chi_0)(1)=1$, i.e., for all $k \in  K_{\mathfrak p}=\SL_2(\OO_{F_{\mathfrak p}})$ we have
$$
\Phi_{\mathfrak p}^0(\tilde g k) = \Phi_{\mathfrak p}^0(\tilde g).
$$
\end{lemma}
When $n$ is odd, $\mathfrak p$ has to be an odd prime of $F$, and there is a
canonical splitting from $K_{\mathfrak p}$ into the metaplectic cover
In the above
identity, we viewed $k$ as an element of the metaplectic cover
via the
canonical splitting.

\subsection{Siegel theta functions and Siegel-Weil formula}
\label{sect:4.3}

A point $z \in \mathbb D$ gives rise to an  orthogonal decomposition
$V_{\sigma_1} =V \otimes_{F, \sigma_1} \mathbb R =z \oplus
z^\perp$.  So an element $x \in V_{\sigma_1}$ can be written as
$x =x_z + x_{z^\perp}$ with respect to this decomposition.  Let
\begin{align}
\phi_1(x, z) &=\phi^{n, 2}(x) = e^{-2 \pi \left(  Q(x_{z^\perp})-  Q(x_{z})\right)},  \quad x \in V_{\sigma_1}, z\in \mathbb D,
\\
\phi_j(x) &= \phi^{n+2, 0}(x) =e^{-2 \pi Q(x)}, \quad x \in V_{\sigma_j},  j >1
\notag
\end{align}
as in Lemma \ref{lem:lambdainfty}. For $x=(x_1,\dots,x_d)\in V(\R)$  let
$$
\phi_\infty(x, z) = \phi_1(x_1 , z) \prod_{j >1} \phi_j(x_j).
$$
Then  Lemma \ref{lem:lambdainfty} implies that
\begin{equation}\label{eq4.5}
\lambda(\phi_\infty(\cdot, z)) =\Phi_\infty^{\tilde \kappa}
\end{equation}
for any $z\in \D$. For a Schwartz function  $\phi_f \in S(V(\hat F))$, we definite its Siegel theta function
of weight $\tilde\kappa$ as
\begin{equation}
\theta(\tau, z, h; \phi_f)
=v^{-\tilde \kappa/2}  \sum_{x \in V(F)} \omega({\tilde g_\tau})\phi_\infty(x, z) \phi_f(h^{-1}x).
\end{equation}
Here $\tilde g_\tau=[n(u) m(\sqrt v), 1] \in \tilde G_\mathbb R$ for $\tau = u + i v \in \mathbb H^d$.  It is a (non-holomorphic)
Hilbert modular form for some congruence subgroup of $\SL_2(F)$. A simple calculation shows that
\begin{align}
\label{theta2}
\theta(\tau,z,h;\phi_f)
&= v_1  \sum_{x\in V(F)}\phi_f(h^{-1} x)
e\big( \tr Q(x_{z^\perp})\tau+ \tr Q(x_{z})\bar\tau\big).
\end{align}
Let
 \begin{align}
\label{theta5}
\Theta_L(\tau,z,h)=\sum_{\mu\in L'/L}\theta(\tau,z,h,\chi_\mu)\chi_\mu
\end{align}
be the $S_L$-valued Siegel theta function associated to $L$. It is a $S_L$-valued  Hilbert modular form of weight $\tilde \kappa$ with representation $\rho_L$, see also \cite[Section 3.3]{Br2}.
Note that compared to \cite{Br2} we have dropped the subscript $S$ from the notation (since there will be only theta functions of this type in the present paper) and have added a subscript $L$ referring to the lattice (since there will be theta functions for different lattices).

Another way to construct such a modular form is  via the coherent Eisenstein series
\begin{align}
E(\tau, s, \phi_f, \tilde\kappa) &=v^{-\tilde\kappa/2} E(\tilde g_\tau, s,
\lambda(\phi_f) \otimes \Phi_\infty^{\tilde\kappa}),
\\
 E_{L}(\tau, s, \tilde\kappa) &=\sum_{\mu \in L'/L}E(\tau, s, \chi_\mu,
 \tilde\kappa)  \chi_\mu.
 \end{align}
Under our assumption on $V$, the Eisenstein series $E_L(\tau, s, \tilde\kappa)$ is holomorphic in $s$ at $s_0$, and its value at $s=s_0$ is a $S_L$-valued Hilbert modular form of weight $\tilde\kappa$ with representation $\rho_L$.

 Let $\Omega$ be the invariant K\"ahler form on $\mathbb D$
 normalized as in \cite{Br2}. Then $\Omega^n$ is an invariant volume form on
 $X_K$. We denote
 $$
\vol(X_K) =\int_{X_K} \Omega^n.
 $$
 If $n\geq 1$ and $C$ is a divisor on $X_K$, we define its degree as
$$
\deg(C)=\int_C \Omega^{n-1}.
$$
 When $n=0$, $\D$ consists of two points $z^\pm$ given by
$V_{\sigma_1}$ with the two possible choices of  an orientation.
Let $\kay =F(\sqrt{-\det V})$ be the quadratic
extension of $F$ associated to $\chi_V$. It is a CM number
field with maximal totally real subfield $F$, and
$H\cong \Res_{k/\mathbb Q} \mathbb G_m$ as algebraic groups. Under this
identification, $K\subset H(\hat{\Q})$ is a compact open subgroup
of $\hat k^\times$, and
\begin{align}
\label{eq:degCM0}
X_K = \kay^\times\backslash \{ z^{\pm} \} \times \hat \kay^\times/K = \{
z^\pm \} \times  \kay^\times\backslash \hat \kay^\times/K
\end{align}
consists of two copies of the idele class group $ \kay^\times\backslash \hat \kay^\times/K$.
For arithmetic reasons, it is convenient to count each point in $X_K$ with  multiplicity $2/w_K$, where
$w_K$ is the number of roots of unity which are contained in $K$ under the embedding $\kay^\times \subset \hat{\kay}^\times$.
In this case,  the `volume' with respect to the counting measure is
\begin{align}
\label{eq:degCM1}
\vol(X_K) = \frac{4}{w_K} |\kay^\times \backslash \hat{\kay}^\times/K|.
\end{align}
When $K=\hat{\OO}_\kay^\times$, then $w_K =w_\kay$ is the number of roots of unity in $\calO_k^\times$ and
\begin{align}
\label{eq:degCM2}
\vol(X_K) = \frac{4}{w_\kay} h_\kay,
\end{align}
where $h_\kay$ is the ideal class number of $\kay$.

\begin{lemma}[Siegel--Weil formula] \label{Siegel-Weil}
Let the notation be as above.  Then
$$
E_{L}(\tau, s_0,  \tilde\kappa) = \frac{c}{\vol (X_{K})}
\int_{X_{K}} \Theta_{L}(\tau, z, h) \Omega^n.
$$
Here $c=1$ for $n \ge 1$, and $c=2$ for  $n=0$.
In particular,  when $n=0$,  $X_K$ is of dimension zero and one has
\begin{align}
\label{eq:kms2}
E_L(\tau, 0, \tilde \kappa)= \frac{2}{\vol (X_K)}\sum_{[z,h]\in X_K} \Theta_{L}(\tau,z,h)
:=\frac{2}{\vol (X_K)}  \sum_{[z, h] \in \supp(X_K)} \frac{2}{w_\kay}\Theta_{L}(\tau,z,h) .
\end{align}
Here each point counts with multiplicity $\frac{2}{w_K}$.
\end{lemma}
\begin{proof}
For a Schwartz function $\phi\in S(V(\A_F))$, let $\theta(\tilde g,t,\phi)$ be the corresponding theta function on $\tilde G_\A\times \SO(V)(\A_F)$.
We fix a base  point $z_0\in \D$ and let $\phi_\infty= \phi_\infty(\cdot,z_0)$.
Taking into account \eqref{eq4.5}, the Siegel--Weil formula (see \cite[Theorem 4.1]{Ku:Integrals}, \cite{KR1}) asserts that
$$
E(\tilde g, s_0, \lambda(\chi_\mu)\otimes\Phi_\infty^{\tilde\kappa})=\frac{c}{\vol([\SO(V)])
}\int_{[\SO(V)]} \theta(\tilde g, t, \chi_\mu \otimes \phi_\infty) dt.
$$
Here $[\SO(V)]=\SO(V)(F) \backslash \SO(V)(\A_F)$, and $dt$ is an
invariant  measure on $[\SO(V)]$ induced by a  Haar measure on $\SO(V)(\A_F)$.

We write $\theta(\tilde g, h, \phi) = \theta(\tilde g, \pi(h), \phi)$ for
$h \in H(\A)$ where $\pi$ is the canonical surjection from $H(\A)=\GSpin(V)(\A_F)$
to $\SO(V)(\A_F)$. Notice that
$$
H(\Q) \backslash H(\A)/K_\infty K \cong X_K,  \quad [
h_\infty h_f]\mapsto [h_{\sigma_1}(z_0), h_f]
$$
is an isomorphism of manifolds, where $K_\infty$ is the stabilizer
of $z_0$ in $H(\R)\cong \prod_j \Gspin(V_{\sigma_j})$.
Then it is
easy to see that
$$
v^{-\tilde\kappa/2} \theta(\tilde g_\tau, h_f h_\infty, \lambda(\chi_\mu
\otimes \phi_\infty) = \theta(\tau, z, h_f, \chi_\mu)
$$
with $z= h_\infty(z_0)$. Now the lemma follows from a change of
variables.
\end{proof}

\subsection{Incoherent Eisenstein series}
\label{sect:4.4}

Here  we consider the  following Eisenstein series of weight $\kappa$ associated to $L$:
\begin{align}
E_L(\tau, s, \kappa) &=\sum_{\mu \in L'/L} E(\tau, s, \mu,\kappa) \chi_\mu,
\\
E(\tau, s, \mu,\kappa) &= v^{-\kappa/2} E(\tilde g_\tau, s,  \lambda( \chi_\mu) \otimes \Phi_\infty^\kappa).
\end{align}
It is incoherent in the sense that it comes from an incoherent quadratic space over $\A_F$ (see Section \ref{sect:global} for details) in a natural way.
Let $\V =\prod_{w\leq \infty} \V_w$ be the `incoherent' quadratic space over
$\A_F$  obtained from $V$ by taking
$\V_w =  V_w$ for every place $w \ne \sigma_1$, and letting
$\V_{\sigma_1}$ be positive definite of dimension $n+2$. In particular, $\hat{\V}=\prod_{w<\infty}\V_w$ is equal to
$\hat{V}$. Let $\LL $ be the image of $\hat{L}$ in
$\hat{\V}$.
Then  $\LL'/\LL \cong L'/L$, and  $\chi_\mu =\Char(\mu + \LL)
\in S(\hat{\V})$. For  $x \in \V_\infty =\prod_j \V_{\sigma_j}$
we let
$$
\phi_{\V, \infty}(x) = \prod_{j=1}^d e^{- 2 \pi Q(x_j)}.
$$
Then Lemma \ref{lem:lambdainfty} implies that
$\lambda(\phi_{\V, \infty}) = \Phi_\infty^\kappa$. So
\begin{equation}
 E(\tau, s, \mu,\kappa) = v^{-\kappa/2} E(\tilde g_\tau, s,
 \lambda(\chi_\mu \otimes \phi_{\V, \infty})).
 \end{equation}

The coherent  Eisenstein series $E_L(\tau, s, \tilde\kappa)$ and the incoherent Eisenstein series $E_L(\tau, s, \kappa)$  are related by the Maass lowing operator in $\tau_1$ as follows (see \cite{Ku:Integrals}, \cite{Br2}):
\begin{align}
\label{eq:eisrel}
L_{\kappa_1}^{(1)} E_L(\tau,s,\kappa)= \frac{1}{2}(s+1-\kappa_1)E_L(\tau,s,\tilde\kappa).
\end{align}
In particular, we have
\begin{align}
\label{eq:eisrel2}
2 L_{\kappa_1}^{(1)}  E_L'(\tau,s_0,\kappa)=E_L(\tau, s_0, \tilde\kappa) .
\end{align}
For the rest of this section, we will study the the derivative $ E_L'(\tau, s_0, \kappa)$.

\subsection{The local Whittaker functions}

For an integer $r \ge 0$ and a finite prime $\mathfrak p$ of $F$,  let $L_r = L_{\mathfrak p} \oplus H^r$, where $H =\OO_{\mathfrak p}^2 $ is the standard quadratic lattice with  with the quadratic form $Q(x, y) =\delta_{\mathfrak p}^{-1}xy$, where $\delta_\mathfrak p$ is a generator of $\partial_\mathfrak p$. Let $V_r = L_r \otimes_{v} F_{\mathfrak p} $ be the associated quadratic space. Notice that $H$ is unimodular (with respect to $\psi_{\mathfrak p}$), and so $L_r'/L_r \cong L_{\mathfrak p}'/L_{\mathfrak p}$. The same calculation as in \cite[Appendix]{Ku:Annals} gives
   \begin{equation}
    W_{m, \mathfrak p}(1, s, \chi_\mu) = \gamma(V_{\mathfrak p}) [L_{\mathfrak p}':L_{\mathfrak p}]^{-\frac{1}2} |\partial_{\mathfrak p}|^{\frac{1}2} W_{\mathfrak p}(s, m, \mu),
   \end{equation}
  where $\gamma(V_{\mathfrak p})$ is an $8$-th root of unity called the local Weil index (see \cite{KuSplit}), and
  \begin{equation} \label{eq4.18}
  W_{\mathfrak p}(r+s_0, m, \mu) =\int_{F_{\mathfrak p}} \int_{\mu + L_r} \psi_{\mathfrak p}(b( Q(x)-m)) dx db
  \end{equation}
  for any integer $r \ge 0$.  Here we normalize the Haar measure such that
   $$
   \vol(L_r, dx) =1, \quad \vol(\OO_{\mathfrak p}, db) =1.
   $$
   The  local Weil indices  have  the product formula:
  \begin{equation} \label{WeilIndex}
   -\prod_{w \le \infty} \gamma(\V_{w}) =\prod_{w \le \infty} \gamma(V_{w}) =1.
   \end{equation}
The same calculation as in \cite[Sections 2-4]{Ya:density}  (see also \cite{KY}) gives the following lemma.

\begin{lemma}
\label{lem:Whittakerfinite}
Let $\mathfrak p$ be a finite prime of $F$.
Then $W_{\mathfrak p}(s, m, \mu) =0$ unless $m \in Q(\mu) + \partial_{\mathfrak p}^{-1}$.
  Assume this condition, and let $\norm(\mathfrak p)$ be the cardinality of the residue field of $F_{\mathfrak p}$.

(1) \quad $W_{\mathfrak p}(s, m, \mu)$  is a polynomial of $\norm(\mathfrak p)^{s_0-s}$ with coefficients in $\mathbb Q$.

(2) \quad If $L_{\mathfrak p}$ is unimodular and $n$ is even, then $\lambda(\chi_0)$ is a spherical section
  in $I(s, \chi_\mathfrak p)$,  i.e., it is $\SL_2(\OO_{\mathfrak p})$-invariant. In this case,  one has
$$
W_{\mathfrak p}(s, m, 0) =\frac{1}{L(s+1, \chi_{\mathfrak p})}
 \sum_{i=0}^{\ord_{\mathfrak p} (m  \partial_{\mathfrak p})} (\chi_{\mathfrak p}(\varpi_{\mathfrak p}) \norm({\mathfrak p})^{-s})^i.
$$
Here $\varpi_{\mathfrak p}$ is a uniformizer of $F_{\mathfrak p}$,  and
$$
W_{\mathfrak p}(s, 0, 0) =
   \frac{L_{\mathfrak p}(s, \chi_{\mathfrak p})}{L_{\mathfrak p}(s+1, \chi_{\mathfrak p})}.
$$

(3)\quad When $L_{\mathfrak p}$ is unimodular and $n$ is odd,  $\lambda(\chi_0)$ is a spherical section
 in $I(s, \chi_\mathfrak p)$. In this case,  one has
 $$
 W_{\mathfrak p}(s, m, 0)= \begin{cases}
 0 &\ff m \notin \partial_{\mathfrak p}^{-1},
 \\
   \frac{L_{\mathfrak p}(s+\frac{1}2, \chi^{(m)}_{\mathfrak p})}{\zeta_{\mathfrak p}(2s+ 1)}b_\mathfrak p(\partial_\mathfrak p m, s+\frac{1}2) &\ff 0\ne  m \in \partial_{\mathfrak p}^{-1} ,
   \\
     \frac{\zeta_{\mathfrak p}(2s)}{\zeta_{\mathfrak p}(2s+1)} &\ff m=0.
     \end{cases}
 $$
 Here $\chi_\mathfrak p^{(m)} =(2\alpha m  , \cdot )_\mathfrak p$ is a quadratic  twist of $\chi_\mathfrak p$ if   $\chi_\mathfrak p=(\alpha, \cdot )_\mathfrak p$, and
 $$
 b_\mathfrak p(\partial_\mathfrak p m, s)= \frac{1-v_\mathfrak p \norm(\mathfrak p)^{-s} + v_\mathfrak p \norm(\mathfrak p)^{k-(1+2k)s} - \norm(\mathfrak p)^{(k+1)(1-2s)}}{1- \norm(\mathfrak p)^{1-2s}}
 $$
with $k=k(m)= [\frac{\ord_\mathfrak p (m \partial_p)}2]$, and
$$
v_\mathfrak p =\begin{cases}
   0 &\ff \ord_\mathfrak p (m \partial_p) \equiv 1 \pmod{2},
   \\
   \chi_\mathfrak p^{(m)}(\pi_\mathfrak p) &\ff \ord_\mathfrak p (m \partial_p) \equiv 0 \pmod{2}.
   \end{cases}
$$

(4) \quad When $V_{\mathfrak p}$ is anisotropic, one has
$$
W_{\mathfrak p}(s_0, 0, \mu)=0.
$$
\end{lemma}

\begin{proof}
We sketch a proof of the formula in (3) and leave the others to the reader. We write $L=L_\mathfrak p$ and $\OO=\OO_\mathfrak p$ in the proof. Let $\delta=\delta_\mathfrak p$ be a uniformizer of $\partial_F$, and $\tilde{\psi}(x) = \psi(\delta^{-1} x)$. Then $\tilde \psi$ is unramified. Let
$\tilde L$ be the lattice $L$ but equipped with the quadratic form $\tilde Q(x) = \delta Q(x)$.  Then $\tilde L$ is unimodular with respect to $\tilde\psi$. So one has
$$
\tilde L \cong \OO^{n+2}, \quad Q(\vec x) = \epsilon  x^2 + \sum_{j=1}^{\frac{n+1}2 }  y_j z_j,
$$
for some $\epsilon \in \OO^\times$. Using this notation, one sees
that $\chi_{\tilde L} =(2 \epsilon, \cdot)_\mathfrak p$. So $\epsilon =
2\alpha \delta \mod (F_\mathfrak p^\times)^2$. Let $\tilde L(0)
=\OO$ with quadratic form $\epsilon x^2$.  Then one has by
definition
\begin{align*}
W_\mathfrak p(r+s_0, m, 0)&= \int_{F_\mathfrak p} \int_{L\oplus
\OO^{2r}} \psi_\mathfrak p(b Q(\vec x) )d\vec x \, \psi_\mathfrak
p(-m b) db
\\
 &=\int_{F_\mathfrak p}\int_{\OO^{n+2 +2r}} \tilde \psi(bx^2)dx \prod_{j=1}^{r+\frac{n+1}2}\tilde\psi(b y_j z_j)  \prod_j dy_j \, dz_j
 \, \tilde{\psi} (-\delta m b) db
 \\
 &:= \tilde{W}_{\mathfrak p} (r+s_0, \delta m,  0).
 \end{align*}
Notice that the last integral  $\tilde{W}_{\mathfrak p} (r+s_0, m,  0)$ is the same normalized Whittaker function as the one defined in (\ref{eq4.18}) but with respect to $\tilde L(0)$ and $\tilde\psi$. It is also  the local density function studied in \cite{Ya:density}  with respect to $\tilde L(0)$ and $\tilde\psi$ when $F_\mathfrak p=\Q_p$. The argument and the formula works for general $F_\mathfrak p$ when $\mathfrak p\nmid 2$. Applying \cite[Theorem 3.1]{Ya:density}, one obtains the formula. Indeed,  the integral is zero if $\delta m \notin \OO$. When $0\ne m \delta \in \OO$,  and
$a:=\ord_\mathfrak p(m \delta) =\ord_\mathfrak p (\partial_\mathfrak p m) \ge 0$ is even, $\chi_\mathfrak p^{(m)}$ is unramified, and  \cite[Theorem 3.1]{Ya:density} gives ($X= q^{-r}$ and $q=\norm(\mathfrak p)$)
\begin{align*}
 \tilde{W}_{\mathfrak p} (r-\frac{1}2, \delta m,  0)
  &= 1 + (1-q^{-1}) \sum_{0 <k \le \frac{a}2} (q X)^{k} + v_{\mathfrak p} q^{\frac{a}2} X^{a+1}
  \\
   &= \frac{L(r, \chi_\mathfrak p^{(m)})}{\zeta_{\mathfrak p}(2r)} b_{\mathfrak p}(\delta m, r)
\end{align*}
as claimed.  When $a$ is odd,  $\chi_\mathfrak p^{(m)}$ is ramified and $L(s, \chi_{\mathfrak p}^{(m)})=1$, and $v_\mathfrak p =0$. In this case,
 \cite[Theorem 3.1]{Ya:density} gives
\begin{align*}
 \tilde{W}_{\mathfrak p} (r-\frac{1}2, \delta m,  0)
  &= 1 + (1-q^{-1}) \sum_{0 < k \le \frac{a-1}2} (q X^2)^{k} - q^{-1} (qX^2)^{a+1}
  \\
   &= (1-X^2) \frac{1- q^{\frac{a+1}2} X^{a+1}}{1- qX^2}
\end{align*}
as claimed.
\end{proof}

  Let
 $$
\Psi(a, b; z)
 =\frac{1}{\Gamma(a)}\int_0^{\infty}
   e^{-z r} (r+1)^{b-a-1} r^{a-1} dr
$$
be the standard confluent hypergeometric function of the second kind \cite{lebedev},
where
$a>0$, $z >0$ and $b$ is any real number. It satisfies the
functional equation \cite[p.~265]{lebedev}
$$
\Psi(a, b; z)
 =z^{1-b} \Psi(1+a-b, 2-b; z).
$$
For convenience, we also define
$$
\Psi(0, b;z)=\lim_{a\rightarrow 0+}  \Psi(a, b; z)=1.
$$
So $\Psi(a, b; z)$ is well-defined for $z>0$, $a \ge \min\{0, b-1\}$.
Finally, for any number $n$, we define
$$
\Psi_n(s, z)
 =\Psi(\frac{1}2(1+ n+s), s+1; z).
$$
Then $(2.7)$ implies
that
$\Psi_n(s, z)=z^{-s} \Psi_n(-s, z)$.
Set
\begin{equation}
W_{\sigma_j}(\tau_j, s, m, \kappa) = \frac{W_{m_j, \sigma_j}(\tau_j, s, \Phi_{\mathbb R}^{\frac{n+2}2})}{\gamma(\V_{\sigma_j})} e(-m_j \tau_j),
\end{equation}
with
$m_j=\sigma_j(m)$.
Then \cite[Proposition 15.1, (15.10) and (15.11)]{KRYComp}   gives the following lemma.

\begin{lemma} \label{lem:Whittakerinfty} \quad
(i) For $m_j > 0$, we have
$$
W_{\sigma_j}(\tau_j, s, m, \kappa)
 =2 \pi\, v_j^\frac{s-s_0}2 \, (2 \pi m)^{s}\, \frac{\Psi_{-s_0 -1}(s, 4 \pi m_j v_j)
}{\Gamma(\frac{s+s_0}2 +1)}.
$$
Moreover,
\begin{align*}
W_{\sigma_j}(\tau_j, s_0, m, \kappa)&=\frac{2 \pi (2\pi m_j)^{s_0}}{\Gamma(s_0+1)},
\\
\frac{W_{\sigma_j}'(\tau_j, s_0, m, \kappa)}{W_{\sigma_j}(\tau_j, s_0, m, \kappa)}&=\frac{1}2 \left[ \log( \pi m_j ) -\frac{\Gamma'(s_0+1 )}{\Gamma(s_0+1 )}
           +J(s_0, 4 \pi m_j v_j)\right],
\end{align*}
where  (for $t >0$, $a \in \mathbb R$)
$$
J(a, t):=\int_{0}^\infty e^{-  t x}
   \frac{(x+1)^{a}-1}{ x}\, dx.
$$

(ii) \quad  For $m_j  <0$, we have
$$
W_{\sigma_j}(\tau_j, s, m, \kappa)
 =2 \pi\,  v^{\frac{s-s_0}2} \, (-2 \pi  m_j)^{s} \, \frac{\Psi_{s_0+1} (s, -4 \pi m_j v_j)
}{\Gamma(\frac{s-s_0}2)}\, e^{4 \pi m_j v_j}.
$$
In particular, $W_{\sigma_j}(\tau_j, s_0, m, \kappa)=0$. Moreover,
$$
W_{\sigma_j}'(\tau_j, s_0, m, \kappa)
 =2^{-s_0} \pi\, v_j^{-s_0}\, \Gamma(-s_0, -4 \pi m_j v_j),
$$
where $\Gamma(a,t)$ denotes the incomplete gamma function.

(iii) \quad For $m=0$,
$$
W_{\sigma_j}(\tau_j, s, m, \kappa)
=2 \pi\, v_j^{-\frac{1}2(s+s_0)}\, \frac{2^{-
s}\Gamma(s)}{\Gamma(\frac{s+s_0}2  +1)\Gamma(\frac{s-s_0}2)}.
$$
In particular, for $ n>0$,  one has $W_{\sigma_j}(\tau_j, s_0, m, \kappa) =0$, and
$$
W_{\sigma_j}'(\tau_j, s_0, m, \kappa)=\frac{\pi}{2^{s_0} s_0} v_j^{-s_0}.
$$
For $n=0$, one  has
$$
W_{\sigma_j}(\tau_j, s, m, \kappa)
=v_j^{-\frac{s}2} \frac{\sqrt\pi\Gamma(\frac{s+1}2)}{ \Gamma(\frac{s+2}2)},
 \quad W_{\sigma_j}(\tau_j, 0, m, \kappa)=\pi .
$$
\end{lemma}

Using the above notation and (\ref{WeilIndex}), we have
\begin{equation} \label{eq: Eisenstein}
E(\tau,s, \mu,\kappa) = \sum_{m \in F} B(v, s, m, \mu) q^m
\end{equation}
where
\begin{equation} \label{eq4.20}
B(v, s, m, \mu) =-\frac{1}{([L':L] D)^{\frac{1}2}}\prod_{\mathfrak p
<\infty} W_{\mathfrak p}(s, m, \mu) \cdot \prod_{j=1}^d W_{\sigma_j}(\tau_j, s,
m, \kappa),
\end{equation}
for $m \ne 0$ and
\begin{equation} \label{eq4.21}
 B(v, s, 0,\mu)= \chi_\mu(0) v^{\frac{1}2 (s-s_0)} -\frac{1}{([L':L] D)^{\frac{1}2}}
 \prod_{\mathfrak p <\infty} W_{\mathfrak p}(s, 0, \mu) \cdot \prod_{j=1}^d W_{\sigma_j}(\tau_j, s, 0, \kappa).
\end{equation}

\subsection{The derivative of $E_L(\tau, s, \kappa)$, the case $n\ge 1$}
\label{sect:4.6}

In this subsection, we assume $n \ge 1$, so $s_0=n/2 >0$.

\begin{proposition}  \label{prop4.7}
Assume that $V$ is anisotropic or has  Witt rank over $F$ less than $n$.
 Then $E(\tau, s_0, \mu,\kappa)$ is a holomorphic Hilbert modular form of weight $\kappa$. Write
$B(m, \mu) =B(v, s_0,  m, \mu)$.

(1) \quad When $ m\gg 0$, one has
$$
B'(v, s_0, m, \mu) = \frac{1}2 B(m, \mu) \sum_{j=1}^d J(s_0, 4 \pi
m_i v_j) + \beta(m, \mu)
$$
for some number $\beta(m, \mu)$ independent of $v$.

(2) \quad When there is a  $j$ with $m_j <0$, one has
$$
B'(v, s_0, m, \mu)=D(m, \mu) v_j^{-s_0} \Gamma(-s_0, -4 \pi m_j v_j)
$$
for some constant $D(m, \mu)$ independent of $v$.  Moreover, $D(m, \mu)=0$ if more than one $m_j <0$.

(3) \quad The constant term is given by $B(0, \mu) =\chi_\mu(0)$ and
$$
B'(v, s_0, 0, \mu) = \frac{1}2 \chi_\mu(0) \log \norm(v) + D(0, \mu) \norm(v)^{-s_0}$$
for some constant $D(0, \mu)$ independent of $v$. Finally, $D(0, \mu)=0$ unless $d=1$ and $V$ is isotropic.
\end{proposition}

\begin{proof}
We first assume that $n$ is even and $m\neq 0$.
Let $S(m)$ be the
set of finite primes $\mathfrak p$ of $F$ for which  $L_{\mathfrak p}$ is not unimodular
or $\ord_{\mathfrak p}(m\partial) >0$. This is a finite set. Lemma
\ref{lem:Whittakerfinite} implies (write
$C=-([L':L]D)^{-1/2}$ temporarily) that
$$
B(v, s, m, \mu)=\frac{C}{ L^{S(m)}(s+1, \chi)} W_{S(m)}(s,m,
\mu) \prod_{j=1}^d W_{\sigma_j}(\tau_j, s, m, \kappa).
$$
Here for a finite set $S$ of finite primes of $F$, we denote
\begin{align*}
L^S(s, \chi) &=\prod_{\mathfrak p \notin S} L_{\mathfrak p}(s, \chi_{\mathfrak p}),
\\
W_S(s, m, \mu) &=\prod_{\mathfrak p \in S} W_{\mathfrak p}(s, m, \mu).
\end{align*}

When $m \gg 0$ is totally positive, Lemma \ref{lem:Whittakerinfty}
implies that
$$
B(m, \mu)=B(v, s_0, m, \mu) = \frac{(2\pi)^{d (s_0+1)}
\norm(m)^{s_0}  C }{\Gamma(s_0+1)^d L^{S(m)}(s_0+1, \chi)}
W_{S(m)}(s_0, m, \mu)
$$
is independent of $v =\Im(\tau)$. If $W_{S(m)}(s_0, m, \mu) =0$,
i.e., $B(m, \mu)=0$, then
$$
B'(v, s_0, m, \mu) =\frac{(2\pi)^{d (s_0+1)} \norm(m)^{s_0}  C
}{\Gamma(s_0+1)^d L^{S(m)}(s_0+ 1, \chi)}W_{S(m)}'(s_0, m, \mu)
$$
is independent of $v$, so (1) holds with $B(m, \mu)=0$ and $\beta(m, \mu) =B'(v, s_0, m, \mu)$.  If $W_{S(m)}(s_0, m, \mu) \ne 0$, i.e.,
$B(m, \mu)\ne 0$, then Lemma \ref{lem:Whittakerinfty} implies
\begin{align*}
\frac{B'(v, s_0, m, \mu)}{B(m, \mu)}
 &= -\frac{L^{S(m), \prime}(s_0+ 1, \chi)}{L^{S(m)}(s_0+1, \chi)}
   + \frac{W_{S(m)}'(s_0, m, \mu)}{W_{S(m)}(s_0, m, \mu)}
    + \sum_j \frac{W_{\sigma_j}'(\tau_j, s_0, m, \kappa)}{W_{\sigma_j}(\tau_j, s_0, m, \kappa)}
    \\
    &= \frac{1}2 \sum_j J(s_0,  4 \pi m_j v_j) +\frac{\beta(m,
    \mu)}{B(m, \mu)}
\end{align*}
for some constant $\beta(m, \mu)$. This proves the case $m \gg 0$.

When $m_j <0$, then $W_{\sigma_j}(\tau_j, s_0, m, \kappa) =0$ and
$$
W_{\sigma_j}'(\tau_j, s_0, m, \kappa)
 =2^{-s_0} \pi\, v_j^{-s_0}\, \Gamma(-s_0, -4 \pi m_j v_j)
$$
by Lemma \ref{lem:Whittakerinfty}. Since $L(s_0+1, \chi) \ne 0$,
the case $m_j <0$ is now clear,

Finally we consider the case $m=0$. The constant term is by Lemma
\ref{lem:Whittakerfinite} and Lemma \ref{lem:Whittakerinfty} given by
$$
E_0(\tau, s, \mu,  \kappa) = \chi_\mu(0) \norm(v)^{\frac{s-s_0}2}  +M(s, \mu) \norm(v)^{-\frac{s+s_0}2},
$$
with
$$
 M(s, \mu)= C   \frac{L^{S}(s,
 \chi)}{L^S(s+1, \chi)} W_{S}(s, 0, \mu)
 \frac{A(s)^d}{\Gamma(\frac{s-s_0}2)^d}.
$$
Here $S$ is the set of finite primes of $F$ where $L_{\mathfrak p}$ is not
unimodular, and
$$
A(s) =\frac{2^{1-s+s_0} \pi \Gamma(s)}{\Gamma(\frac{s+s_0}2)}.
$$
%
So
$$
E_0(\tau, s_0, \mu, \kappa) =\chi_\mu(0) +M(s_0, \mu) \norm(v)^{-s_0},
$$
and
$$
E_0'(\tau, s_0, \mu, \kappa) = \frac{1}2 (\chi_\mu(0) +M(s_0, \mu)) \log \norm(v)
   + M'(s_0, \mu) \norm(v)^{-s_0}.
$$
Now we have to analyze $\ord_{s=s_0} M(s, \mu)$ case by case.

{\bf Case 1}: When $n>2$ or $n=2$ with $\chi$ non-trivial, the $L$-functions in $M(s, \mu)$ have no zeros or  poles at $s_0$, so
$$
\ord_{s=s_0} M(s_0, \mu) \ge d.
$$
Consequently, $M(s_0, \mu) =0$.  Moreover  $M'(s_0, \mu) =0$ if  $d>1$. When $d=1$ and $V$ is anisotropic, $V$ has to be anisotropic at
some finite prime $\mathfrak p$ (since it is isotropic at infinity). This implies
$$
W_\mathfrak p(1, s_0, \mu) =0,
$$
and
$$
\ord_{s=s_0} M(s, \mu) \ge 1+1 =2.
$$

{\bf Case 2}: When $n=2$ and $\chi$ is trivial, $L^S(s, \chi) =\zeta^S(s)$ has a simple pole at $s=s_0=1$, and one has also
$$
\ord_{s=s_0} M(s_0, \mu) = d-1 + \ord_{s=s_0} W_S(s, 0, \mu).
$$
When $d>2$, one has again $M(s_0, \mu) =M'(s_0, \mu) =0$.

Let $C^+(V)$ be the even Clifford algebra of $V$ over $F$. Then $C^+(V) =B \times B$ for a quaternion algebra $B$ over $F$
by \cite[Lemma 0.2]{KRCrelle}, since the quadratic extension $\kay/F$ associated to $\chi_V=\chi$ is $F \times F$. Furthermore,
\cite[Lemma 0.3]{KRCrelle} implies that
$$
(V, Q) \cong (B, \alpha \det)
$$
for some $\alpha \in F^\times$, where $\det $ is the reduced norm of $B$.
Our assumption on the signature of $V$ implies that $B$ is split at
$\sigma_1$,   and ramified at the other  infinite primes.

If $d=2$, then there is at least one finite prime $\mathfrak p \in S$ such that $B$ is ramified, i.e, $V_{\mathfrak p}$ is
anisotropic. This implies
$$
W_{\mathfrak p}(s_0, 0, \mu) =0, \quad W_S(s_0, 0, \mu) =0,
$$
and consequently
$$\ord_{s=s_0}M(s, 0, \mu) \ge 2.
$$

If $d =1$, then $B$ is either $M_2(\mathbb Q)$ or a division quaternion algebra. But $B=M_2(\mathbb Q)$  implies that the Witt rank of
$V$ is $2$, contradicting our assumption. So $B$ is an indefinite division algebra, and there are at least two finite primes
$p$ and $q$ at which $V$ is anisotropic. This implies
$$
W_p(s_0, 0, \mu) =W_q(s_0, 0, \mu) =0,
$$
and so
$$
\ord_{s=s_0} M(s, \mu) =\ord_{s=s_0} W_S(s, 0, \mu) \ge 2.
$$
In summary we have proved for $n$ even and under our assumption on $V$ that $M(s_0, \mu)=0$. Hence
$$
E_0(\tau, s_0, \mu, \kappa) =\chi_\mu(0)
$$
is holomorphic, and thus $E(\tau, s_0, \mu, \kappa)$ is holomorphic.  Moreover,
$$
E_0'(\tau, s_0, \mu, \kappa) =\frac{1}2 \chi_\mu(0) \log \norm(v) + D(0, \mu) \norm(v)^{-s_0}
$$
with $D(0, \mu):=M'(s_0, \mu) =0$ unless $d=1$ and $V$ is isotropic.  This proves the assertion when $n$ is even.

Next, we assume that $n$ is odd. The proof  is similar  with a  slight change due to the different $L$-functions showing up
 in the formulas in Lemma \ref{lem:Whittakerfinite}. We deal with the constant term which is trickier and leave the other cases
 to the reader. By Lemma \ref{lem:Whittakerfinite} and Lemma \ref{lem:Whittakerinfty} we have
$$
E_0(\tau, s, \mu,  \kappa) = \chi_\mu(0) \norm(v)^{\frac{s-s_0}2}  +M(s, \mu) \norm(v)^{-\frac{s+s_0}2},
$$
with
$$
 M(s, \mu)= C_1   \frac{\zeta^{S}(2s)}{\zeta^S(2s+1)} W_{S}(s, 0, \mu)
 \frac{A(s)^d}{\Gamma(\frac{s-s_0}2)^d},
$$
for some constant $C_1$. Here $A(s)$ and $S$ are the same as above. The case $n >1$  or $d >2$ is clear, notice that $V$ is always
isotropic when $d=1$ and $n \ge 3$.
Now assume that $n=1$ and  $d =1, 2$. In this case, $\zeta^S(2s)$ has a simple pole at $s=s_0=\frac{1}2$, and so
$$
\ord_{s=s_0} M(s, \mu) = d-1 + \ord_{s=s_0} W_S(s, 0, \mu).
$$
Since $\dim V=3$, $C^+(V)=B$ is a quaternion algebra over $F$ and $(V, Q) \cong (B^0 , \alpha \det )$ for some $\alpha \in F^\times$,
where $B^0$ is the subspace of $B$ of trace zero elements.

When $d=2$, $B$ is split  at $\sigma_1$ and ramified at $\sigma_2$. So $B$ has to be ramified at some finite prime $\mathfrak p$, and
$V$ is anisotropic at $\mathfrak p$. This implies that $\mathfrak p \in S$, and
$$
W_\mathfrak p(s_0, 0, \mu)=0, \quad  W_S(s_0, 0, \mu) =0.
$$
So
$
\ord_{s=s_0} M(s, \mu) \ge 2
$.

When $d=1$, $B$ has to be an indefinite division quaternion over $\Q$ since $V$ is anisotropic by our assumption in this case.
So $B$ has to be ramified at least at two finite primes $p$ and $q$, i.e., $V$ is anisotropic at $p$ and $q$. So
$$
\ord_{s=s_0} M(s, \mu) =\ord_{s=s_0} W_S(s, 0, \mu) \ge 2.
$$
This proves that
$$
E_0(\tau, s_0, \mu, \kappa) =\chi_\mu(0)
$$
is holomorphic, and
$$
E_0'(\tau, s_0, \mu, \kappa) =\frac{1}2 \chi_\mu(0) \log\norm(v) + D(0,\mu) \norm(v)^{-s_0}
$$
with $D(0, \mu) =M'(s_0, \mu) =0$ unless $d=1$ and $V$ is isotropic.
\end{proof}

\begin{corollary}
Let the notation and assumption be as in  Proposition \ref{prop4.7}.   Write
$$
\mathcal E_L(\tau) =E_L'(\tau, s_0, \kappa).
$$
Then we have
$$
\mathcal E_L(\tau) =\sum_{j=0}^{d} \mathcal E_L^{(j)}(\tau) ,
$$
where, for $ j >0$,
\begin{align*}
\calE_L^{(j)}(\tau)&:= \frac{1}{2}\log(v_j)\chi_0+ \norm(v)^{-s_0} \sum_{\mu \in L'/L} D(0, \mu) \chi_\mu \notag
\\
&\phantom{=} +\frac{1}2 \sum_{\mu\in L'/L}\sum_{\substack{m\in Q(\mu)+ \partial^{-1}\\m\gg0}}
B_L(m,\mu)J(s_0,4\pi m_j v_j)q^m\chi_\mu   \\
&\phantom{=} +\sum_{\mu\in L'/L}\sum_{\substack{m\in Q(\mu) + \partial^{-1}\\m_j<0\\ \text{$m_i>0$ for $i\neq j$}}}
D(m, \mu) v_j^{-s_0} \Gamma(-s_0, - 4 \pi m_j v_j)  q^m\chi_\mu, \notag
\end{align*}
is holomorphic in all $\tau_i$ with $i \ne j$ (but non-holomorphic in
$\tau_j$),  and
\begin{align*}
\calE_L^{(0)}(\tau):= \calE_L(\tau)-\calE_L^{(1)}(\tau)-\dots-\calE_L^{(d)}(\tau) +(d-1) \norm(v)^{-s_0} \sum_{\mu \in L'/L} D(0, \mu) \chi_\mu.
\end{align*}
 is holomorphic in $\tau$. Finally, $D(0, \mu)=0$ unless $d=1$ and $V$ is isotropic.
\end{corollary}

\begin{proof}
Everything is immediate from Proposition \ref{prop4.7} except the claims about holomorphicity.
One has by Proposition  \ref{prop4.7} that
$$
\calE_L^{(0)}(\tau)= \sum_{\mu \in L'/L}\sum_{\substack{m \gg 0 \\ \mu \in L'/L}}\beta( m, \mu) q^m \chi_\mu,
$$
which is visibly holomorphic.

  When $d \ge 2$,  one has $D(0, \mu)=0$, and thus $\calE^{(j)}(\tau)$ is holomorphic in $\tau_i$ for $i\neq j$.
  When $d=1$, the claim on $\calE^{(j)}$ is empty for $j=1$.
\end{proof}

\subsection{The derivative of $E_L(\tau, s_0, \kappa)$, the case $n=0$}

\label{sect:4.7}

The case $n=0$ is special and has a particularly nice formula. For a non-zero $m \in F$, let $\Diff(\V, m)$ be the set of primes $w$ of $F$ such that
$\V_w$ does not represent $m$, following Kudla \cite{Ku:Annals} ($\mathbb V$ is defined in Section  \ref{sect:4.4}).
 Then Kudla proved that $|\Diff(\V, m)|$ is finite and odd, and for every $w \in \Diff(\V, m)$,
 the local Whittaker function $W_{\mathfrak p}(s, m, \mu)$ (when $w=\mathfrak p <\infty$) or
  $W_{\sigma_j}(\tau_j, s, m, \kappa)$ (when $w =\sigma_j$) vanishes at the center $s_0=0$. Let $\kay$ be the totally imaginary quadratic extension of $F$ associated to $\chi=\chi_{\V} =\chi_V$. It is easy to check that $w \in \Diff(\V, m)$ implies that $w$ is non-split in $\kay$. We normalize
\begin{align}
W_{m, \mathfrak p}^*(s,  \chi_\mu) &=|\partial D_{\kay/F}|_{\mathfrak p}^{-\frac{1}2}  L_{\mathfrak p}(s+1, \chi_{\mathfrak p})
  \frac{W_{m, \mathfrak p}(1, s, \chi_\mu)}{\gamma(\V_\mathfrak p)}
\\
 &=\frac{L(s+1, \chi_\mathfrak p)}{([L_\mathfrak p':L_\mathfrak p] | D_{\kay/F}|_\mathfrak p)^{\frac{1}2}} W_\mathfrak p(s, m, \mu), \notag
 \end{align}
and
\begin{align}
W_{m, \sigma_j}^*(\tau_j, s,  \kappa)
 &=L_{\sigma_j}(s+1, \chi_{\sigma_j}) \frac{W_{m, \sigma_j}(\tau_j, s,  \Phi_{\sigma_j}^+)}{\gamma(\V_{\sigma_j})} e(-m_j \tau_j) \\
 \nonumber
 &= L_{\sigma_j}(s+1, \chi_{\sigma_j})W_{\sigma_j}(\tau_j, s, m, \kappa).
\end{align}
Here $D_{\kay/F}$ is the relative discriminant of $\kay/F$ (we also denote $\partial_{\kay/F}$ for the relative different), and
$$
L_{\sigma_j}(s, \chi_{\sigma_j}) = \pi^{-\frac{s+1}2} \Gamma(\frac{s+1}2).
$$
 Then we have the following result due to Howard and the second author  \cite[Theorem 6.2.7]{HY}.

 \begin{lemma} \label{lem: integral} Let the notation be as above.

 (1) \quad  One has
 $W_{m, \mathfrak p}^*(0,  \chi_\mu)  \in \mathbb Q$, and  $W_{m, \mathfrak p}^{*, \prime}(0, \chi_\mu) \in \mathbb Q \log \norm(\mathfrak p)$.

 (2) \quad One has
 $W_{m, \mathfrak p}^*(0,  \chi_\mu)  \in \mathbb Z$  unless   $\mu \notin L_\mathfrak p$  and $\mathfrak  p | 2$.

 (3) \quad When  $(L_\mathfrak p, Q) \cong (\OO_{\kay, \mathfrak p}, \xi_\mathfrak p x \bar x)$ with $\ord_\mathfrak p \xi_\mathfrak p =- \ord_\mathfrak p \partial$, one has
 $W_{m, \mathfrak p}^{*, \prime}(0, \chi_\mu) \in \mathbb Z\log \norm(\mathfrak p)$,  unless  $\mu \notin L_\mathfrak p$  and  $ \mathfrak  p | 2$.
\end{lemma}

\begin{proof}
Let $\delta \in F_\mathfrak p^\times $ with $\delta \OO_\mathfrak p = \partial_{\mathfrak p}$. Let $\tilde\psi_\mathfrak p(x) = \psi_{\mathfrak p}(\delta^{-1} x)$, and let $\tilde{\V}_\mathfrak p = \V_\mathfrak p$ with a different quadratic form $\tilde Q(x) = Q(\delta x)$. Then \begin{equation} \label{eq:psi}
\omega_{\V_\mathfrak p, \psi_\mathfrak p} = \omega_{\tilde{\V}_\mathfrak p, \tilde{\psi}_\mathfrak p}.
\end{equation}
If $\phi \in S(\V_\mathfrak p)$, we view it also as a function in $S(\tilde{\V}_\mathfrak p)$. Then their images in $I(0, \chi_\mathfrak p)$ are the same, and thus the associated Whittaker functions are related via
\begin{equation}
|\delta|_\mathfrak p^{\frac{1}2} W_{\delta m,  \mathfrak p}^{\tilde{\psi}_\mathfrak p}(g, s, \phi) =  W_{m,  \mathfrak p}^{\psi_\mathfrak p}(g, s, \phi).
\end{equation}
Here we use the super script $\psi$ to indicate the dependence of the local Whittaker functions on $\psi$. So one has
\begin{equation} \label{eq: Whittaker-psi}
W_{m, \mathfrak p}^{*, \psi_\mathfrak p}(s, \chi_\mu) =W_{\delta m, \mathfrak p}^{*, \tilde{\psi}_\mathfrak p}(s, \chi_\mu)
\end{equation}
where the right hand side is the normalization  in\cite[Section 6]{HY}. Now our lemma is just \cite[Theorem 6.2.7]{HY}.
\end{proof}

Let $S=S(L)$ be the set of finite primes of $F$ such that $L_{\mathfrak p}$
is not unimodular (with respect to $\psi_{\mathfrak p}$), and for a non-zero $m
\in F$ let $S(m)$ be the union of  $S$ and the set of
finite primes of $F$ dividing $m\partial$. Then Lemma
\ref{lem:Whittakerfinite} and Lemma \ref{lem:Whittakerinfty} imply that
\begin{align} \label{eq: WhittakerValue}
 W_{m, \mathfrak p}^*(0,  \chi_\mu) &=1   &&\ff\mathfrak p \notin S(m), \mathfrak p<\infty, \notag
 \\
 W_{m, \sigma_j}^*(\tau_j, 0,  \kappa) &=2  && \ff m_j >0,
 \\
 W_{m, \sigma_j}^{*, \prime}(\tau_j, 0,  \kappa) &=\Gamma(0, -4 \pi m_j v_j)   && \ff m_j <0. \notag
\end{align}
Let $A= \norm_{F/\Q}(\partial D_{\kay /F})=D\norm_{F/\Q}( D_{\kay /F}) $ where
$D_{\kay/F}$ is the relative discriminant of $\kay/F$, and let
$$
\Lambda(s, \chi) =A^{\frac{s}2} \prod_{\mathfrak p < \infty} L_{\mathfrak p}(s, \chi_{\mathfrak p})
  \prod_{j=1}^d L_{\sigma_j}(s, \chi_{\sigma_j})
$$
be the complete $L$-function of $\chi$. Then $\Lambda(s, \chi)=\Lambda(1-s, \chi)$ and
$\Lambda(1, \chi)\in \mathbb Q$.
Finally, define
\begin{align} \label{eq4.26}
W^{(\infty)}(m, \mu) &=\prod_{\mathfrak p\in S(m)} W_{m \mathfrak p}^*(0, \mu),
\\
W^{(\mathfrak p \infty)}(m, \mu) &=\prod_{\mathfrak p'\in S(m), \mathfrak p'\ne \mathfrak p} W_{m, \mathfrak p}^*(0, \mu). \notag
\end{align}

\begin{proposition} \label{prop4.9} \label{prop: n=0}
Let the notation and assumption be as above.  Then  $E_L(\tau, 0,  \kappa)=0$. Write for $\mu \in \LL'/\LL=L'/L$ as in (\ref{eq: Eisenstein}):
$$
E(\tau, s, \mu,  \kappa) =\sum_{m \in Q(\mu) + \partial^{-1}} B(v, s, m, \mu) q^m \chi_\mu.
$$

(1) \quad  When $\Diff(\V, m) = \{ \mathfrak p\}$  consists of a unique finite prime $\mathfrak p$  of $F$, then
$$
B'(v, 0, m, \mu) = -\frac{2^{d}\log \norm(\mathfrak p)}{\Lambda(1, \chi)} W^{(\mathfrak p \infty)}(m, \mu) \frac{W_{m, \mathfrak p}^{*, \prime}(0, \mu)}{\log \norm(\mathfrak p)}.
$$
Moreover,
$$
\frac{W_{m, \mathfrak p}^{*, \prime}(0, \mu)}{\log \norm(\mathfrak p)} = \frac{1}2 (1 + \ord_\mathfrak p (m \partial))
$$
if $\mathfrak p \notin S(m)$.

(2) \quad  When $\Diff(\V, m) = \{ \sigma_j \}$, one has in particular  $m_j =\sigma_j (m) <0$ and $m_l >0$ for $l \ne j$, and
$$
B'(v, 0, m, \mu) = -\frac{2^{d-1}}{\Lambda(1, \chi)} W^{(\infty)}(m, \mu) \Gamma(0, - 4 \pi m_j v_j).
$$

(3) \quad When  $|\Diff(\V, m) | >1$, one has $B'(v, 0, m, \mu) =0$.

(4) \quad One has
$$
B'(v, 0, 0, \mu) = \chi_\mu(0) \log N(v) + \beta(0, \mu)
$$
for some constant $\beta(0, \mu)$.
\end{proposition}

\begin{proof}
We have already shown $B(v,s, m, \mu) =0$ unless $m \in Q(\mu) +\partial^{-1}$. So we assume this condition throughout the proof.
Since $w \in \Diff(\V, m)$ implies $W_{m, w}(\tau, 0,
\mu,\kappa) =0$, one sees immediately that $B(\tau, 0,
m, \mu) =0$ for all $m \ne 0$.
Thus we have
$$
E(\tau, 0, \mu,\kappa)=B(\tau, 0, \mu, \kappa) =0$$ since it is
of weight $\kappa=(1, \cdots, 1)$.  Moreover, it implies
$B'(\tau, 0, m, \mu)=0$ if $|\Diff(\V, m)| >1$ for
$m\ne 0$. Since
$$
B(v, s, 0, \mu)= \chi_\mu(0) \norm(v)^{\frac{s}2} + \norm(v)^{-\frac{s}2} f(s)
$$
for some function $f(s)$ which is holomorphic at $s=0$, we see $f(0) =-\chi_\mu(0)$, and
$$
B'(v, 0, 0, \mu) = \chi_\mu(0) \log \norm(v) + f'(0).
$$
This proves (3) and (4).  To prove the other assertions, we notice for $m \ne 0$ that
\begin{align*}
B(v, s, m, \mu) &= \frac{|A|^{\frac{s}2}}{\Lambda(s+1, \chi)} \prod_{w \le \infty} \gamma(\V_w) \prod_{\mathfrak p} W_{m, \mathfrak p}^*(s, \mu)  \prod_{j} W_{m, \sigma_j}^*(\tau_j, s, \kappa)
\\
 &=-\frac{|A|^{\frac{s}2}}{\Lambda(s+1, \chi)}\prod_{\mathfrak p} W_{m, \mathfrak p}^*(s, \mu)  \prod_{j} W_{m, \sigma_j}^*(\tau_j, s, \kappa).
 \end{align*}
 Here we used (\ref{WeilIndex}).
Now (1) and (2) follow from (\ref{eq: WhittakerValue}) and (\ref{eq4.26}).
\end{proof}

\begin{corollary}
\label{cor4.10} \label{cor: n=0}
Let the notation and assumption be as in  Proposition \ref{prop4.9}. Write
$$
\mathcal E_L(\tau) =E_L'(\tau, s_0, \kappa).
$$
Then we have
$$
\mathcal E_L(\tau) =\sum_{j=0}^{d} \mathcal E_L^{(j)}(\tau).
$$
Here for  $ j >0$,
\begin{align*}
\calE_L^{(j)}(\tau)&:=\chi_\mu(0)  \log v_j -\frac{2^{d-1}}{\Lambda(1, \chi)} \sum_{\mu\in L'/L}\sum_{\substack{m\in Q(\mu) + \partial^{-1}\\m_j<0\\ \text{$m_i>0$ for $i\neq j$}}}
 W^{(\infty)}( m, \mu) \Gamma(0, - 4 \pi m_j v_j)  q^m \chi_\mu,
\end{align*}
is holomorphic in $\tau_i$ for all $i \ne j$ (but non-holomorphic in $\tau_j$),  and
\begin{align*}
\calE_L^{(0)}(\tau)
 &:=   \sum_{\mu \in L'/L} \beta_L(0, \mu) \chi_\mu    - \frac{2^d}{\Lambda(1, \chi)}\sum_{\substack{\mu \in L'/L, m\gg 0\\ m \in Q(\mu) + \partial^{-1}} } \beta_L^*(m, \mu) q^m \chi_\mu
\end{align*}
is holomorphic in $\tau$. Here
for  $ m \gg 0$ we have
$$
\beta_L^*(m, \mu) = \begin{cases} W^{(\mathfrak p \infty)}(m, \mu) W_{m,\mathfrak
p}^{*, \prime}(0, \mu) &\ff  \Diff(\V, m) =\{ \mathfrak p\},
\\
 0 &\hbox{otherwise.}
 \end{cases}
$$
In particular
$$
\beta_L^*(m, \mu) =\begin{cases}
   a_{\mathfrak p} \log \norm(\mathfrak p) &\ff \Diff(\V, m) =\{ \mathfrak p\},
   \\
    0 &\hbox{otherwise},
    \end{cases}
$$
for some constant $a_{\mathfrak p} \in \mathbb Q$. Furthermore    $a_{\mathfrak p} =0$
unless  $\mathfrak p$ is non-split in $\kay$ satisfying the additional
condition  that $\ord_{\mathfrak p} m\partial >0$ is odd   or $\mathfrak p \in S(L)$.
\end{corollary}

\begin{remark}
\label{rem:E_L}
Notice that the incoherent Eisenstein series
$E_L(\tau, s, \kappa)$ depends only on $\LL \subset \hat{\V}$. For
this reason, we will sometimes write $E_\LL $ for $E_L$,  $\mathcal E_\LL$
for $\mathcal E_L$, and $\mathcal E_\LL^{(j)}$ for $\mathcal
E_L^{(j)}$.
\end{remark}

\begin{example}
\label{ex4.13} Let $\kay$ be a totally imaginary quadratic extension
of $F$ which is unramified at primes of $F$ above $2$, and let
$\chi=\chi_{\kay/F}$ be the quadratic Hecke character of $F$
associated to $\kay/F$. Suppose  that $\xi_\A =\hat{\xi} \xi_\infty
\in \A_F^\times=\hat{F}^\times F_{\infty}^\times$ satisfies
\begin{equation} \label{eq: beta}
\hat{\xi} \hat{\OO}_F = \hat{\partial}^{-1}, \quad
\xi_j:=\xi_{\sigma_j} >0, \quad  \chi(\xi_\A) =-1.
\end{equation}
Then $(\V, Q) = (\kay_\A,   \xi_\A z \bar z)$ is an incoherent
quadratic space over $\A_F$, which is positive definite at all
infinite places, and  $\chi_{\V} =\chi$.  It is easy to check that
$\mathfrak p \in  \Diff(\V, m)$ if and only if $\chi_\mathfrak
p(\xi_\mathfrak p m) =-1$.

Let $\LL =\hat{\OO}_k$, then
$$
\LL'/\LL \cong \partial_{\kay/F}^{-1}/\OO_\kay.
$$
Let
 $$E_\LL(\tau, s, \kappa) =\sum_{\mu \in  \partial_{\kay/F}^{-1}/\OO_\kay} E(\tau, s, \mu, \kappa)\chi_\mu
 $$
 be the incoherent Eisenstein series  in Proposition \ref{prop: n=0}, and write
 $$
 \mathcal E_\LL^{(0)}(\tau)= \sum_{\mu \in  \partial_{\kay/F}^{-1}/\OO_\kay} \beta_\LL(0, \mu) \chi_\mu - \frac{2^d}{\Lambda(1, \chi)}\sum_{\substack{\mu \in  \partial_{\kay/F}^{-1}/\OO_\kay, m \gg 0 \\ m \in Q(\mu) + \partial^{-1}}} \beta_\LL^*(m, \mu) q^m \chi_\mu
 $$
as in Corollary \ref{cor: n=0}.
Then we have:

\begin{proposition}
\label{prop:ex}
Let the notation and assumption be as
above. Let $o(\mu)$ be the number of prime ideals $\mathfrak l$
of $F$ which ramify in  $\kay$ and for which $\mu_{\mathfrak l} \in
\OO_{\kay, \mathfrak l}$.
Moreover, for an ideal $\mathfrak a$ of $F$, we let
$$
\rho (\mathfrak a) = |\{ \mathfrak A \subset \OO_\kay;\;   \norm_{\kay/F}(\mathfrak A) =\mathfrak a\}|.
$$
Assume that $ m \in F$ is totally positive.
If $|\Diff(\V, m) |
>1$ or $m \notin Q(\mu) +
\partial^{-1}$, we have  $\beta_\LL^*(m, \mu) =0$.

If $\Diff(\V, m) =\{ \mathfrak  p \}$ and $m \in
Q(\mu) +
\partial^{-1}$, then $\mathfrak p$ is
non-split in $\kay$, and $\beta_\LL^*(m, \mu) \in \mathbb Z \log
\norm(\mathfrak p)$ is given by
$$
\beta_\LL^*(m, \mu)=2^{o(\mu) -1} (1+ \ord_\mathfrak p(m
\partial)) \times
\begin{cases} \rho(m \partial D_{\kay/F} \mathfrak p^{-1}) \log
\norm(\mathfrak p), & \text{if $\mathfrak p$ is inert in $\kay$,}\\
\rho(m \partial D_{\kay/F} ) \log \norm(\mathfrak p),
& \text{if $\mathfrak p$ is ramified in $\kay$.}
\end{cases}
$$
\end{proposition}

\begin{proof}
This is a  generalization of  \cite[Theorem 4.1]{Scho} and \cite[Theorem 2.6]{BY} and follows from Proposition  \ref{prop: n=0} and local results in \cite{YaValue} and \cite[Section 6]{HY}. Indeed, using the notation in the proof of Proposition  \ref{lem: integral}, we have
$$
W_{m, \mathfrak p}^*(s, \chi_\mu) = W_{\delta m, \mathfrak p}^{*, \tilde{\psi}_\mathfrak p}(s, \chi_\mu)
$$
by (\ref{eq: Whittaker-psi}), which is zero unless $m\delta  \in D_{\kay/F, \mathfrak p}^{-1}$. Here $\delta$     is an $\OO_{\mathfrak p}$-generator of  $ \partial_{\mathfrak p}$, and  $\tilde\psi_\mathfrak p(x) =\psi_\mathfrak p(\delta x)$ is unramified. Moreover, $(\tilde{\LL}_\mathfrak p, \tilde Q) = (\OO_{\kay, \mathfrak p}, \delta \xi_\mathfrak p z \bar z)$ with
$ \delta \xi_\mathfrak p \in \OO_\mathfrak p^\times$.  Denote $N=\ord_\mathfrak p (m \partial) = \ord_\mathfrak p (m \delta)$,  let $\pi_\mathfrak p$ be a uniformizer of $F_\mathfrak p$, and assume $m \in  (\partial_\mathfrak p D_{\kay/F, \mathfrak p})^{-1}$.

When $\mathfrak p$ is unramified in $\kay$, \cite[Proposition 1.1]{YaValue}  (also \cite[Proposition 6.2.2]{HY}) gives
$$
W_{m, \mathfrak p}^*(s, \chi_\mu) =\sum_{n=0}^N (\chi_\mathfrak p (\pi_\mathfrak p) \norm(\mathfrak p)^{-s})^n.
$$
 Consequently,
 $$
 W_{m, \mathfrak p}^*(0, \chi_\mu) = \sum_{n=0}^N (\chi_\mathfrak p (\pi_\mathfrak p))^n =\rho_\mathfrak p(m \partial )
 $$
 is the number of ways to write $m \partial$ as the norm of  integral ideals of $\kay_\mathfrak p$ (to $F_\mathfrak p$). In particular,
 it vanishes if and only if $\mathfrak p \in \Diff(\V, m )$, that is,  $\chi_\mathfrak p(\xi_\mathfrak p m) =\chi_\mathfrak p (\delta  m) =-1$. In such a case, one has
 $$
 W_{m, \mathfrak p}^{*, \prime}(0, \chi_\mu)= \frac{1}2 (1 + \ord_\mathfrak p (m\partial)) \log \norm(\mathfrak p).
 $$

When $\mathfrak p$ is ramified in $\kay$, and $\mu_\mathfrak p \in \LL_\mathfrak p=\OO_{\kay, \mathfrak p}$, one has by \cite[Proposition 1.3]{YaValue}  (see also \cite[Proposition 6.2.4]{HY}) that
$$
W_{m, \mathfrak p}^*(s, \chi_\mu) = 1 + \chi_{\mathfrak p}(\xi_\mathfrak p m) \norm(\mathfrak p)^{-s(f_\mathfrak p +N)},
$$
where $f =\ord_{\mathfrak p} D_{\kay/F}$. So
$$
W_{m, \mathfrak p}^*(0, \chi_\mu) = \begin{cases}
 0 &\ff \mathfrak p \in \Diff(\V, m),
 \\
 2 &\ff \mathfrak p \notin \Diff(\V, m).
 \end{cases}
$$
When $\mathfrak p \in \Diff(\V, m)$, one has
$$
W_{\mathfrak p}^{*, \prime}(0, \chi_\mu) = (f + \ord_\mathfrak p m \partial) \log \norm(\mathfrak p).
$$
Finally, when $\mathfrak p \nmid 2 $ is ramified in $\kay$, and
$\mu_\mathfrak p \notin \LL_\mathfrak p=\OO_{\kay, \mathfrak p}$, one
has by \cite[Proposition 6.2.6]{HY} that
$$
W_{m, \mathfrak p}^*(s, \chi_\mu) = \hbox{Char}( Q(\mu) + \partial_\mathfrak p^{-1})(m) =1.
$$

Now our formula follows pretty easily.  Indeed, if $|\Diff(\V, m)|
>1$ or $m \notin Q(\mu) + \partial^{-1}$, then $\beta_\LL^*(m, \mu) =0$.
So we may assume that $\Diff(\V, m) =\{ \mathfrak  p\}$ and $m \in
Q(\mu) + \partial^{-1}$. This implies in particular that  $\mu \in
\OO_{\kay, \mathfrak p}$ and $\mathfrak p$ is non-split in $\kay$.

When $\mathfrak p$ is inert in $\kay$, the above formulas and
Proposition \ref{prop: n=0} imply that
\begin{align*}
\beta_\LL^*(m, \mu) &=2^{o(\mu) -1} (1 + \ord_{\mathfrak p}(m
\partial)) \prod_{\mathfrak l \nmid  \mathfrak p D_{\kay/F}}
\rho_\mathfrak l(m \partial) \log \norm(\mathfrak p)
\\
 &= 2^{o(\mu) -1} (1 + \ord_{\mathfrak p}(m \partial))  \rho(m \partial \mathfrak p^{-1}) \log \norm(\mathfrak p)
\end{align*}
as claimed. Here we used the fact that for an integral ideal
$\mathfrak a$ of $F$
\begin{align*}
\rho(\mathfrak a) &=\prod_{\mathfrak l} \rho_\mathfrak l(\mathfrak a),
\end{align*}
and
$$
   \rho_\mathfrak l(\mathfrak a) = \begin{cases}
    1 &\ff \mathfrak l  \hbox{ is ramified in } \kay,
    \\
    \frac{1+ (-1)^{\ord_\mathfrak l \mathfrak a}}2 &\ff \mathfrak l  \hbox{ is inert in } \kay,
    \\
    1 +\ord_{\mathfrak l} \mathfrak a &\ff \mathfrak l  \hbox{ is split in } \kay.
    \end{cases}
$$

When $\mathfrak p$ is ramified in $\kay$, the above formulas and
Proposition \ref{prop: n=0} imply that
\begin{align*}
\beta_\LL^*(m, \mu) &=2^{o(\mu)-1 } (f_\mathfrak p + \ord_{\mathfrak
p}(m \partial)) \prod_{\mathfrak l \ne \mathfrak p} \rho_\mathfrak
l(m \partial) \log \norm(\mathfrak p)
\\
 &=2^{o(\mu)-1 } (1 + \ord_\mathfrak p ( m \partial) ) \rho(m \partial) \log \norm(\mathfrak p)
\end{align*}
as claimed, since $f_\mathfrak p=1$ under our assumption that $\kay/F$ is unramified at primes above $2$. This proves the proposition.
\end{proof}
\end{example}

\section{Automorphic Green functions}

\label{sect:green}
\label{sect:5}

Here we briefly recall from \cite{Br2} the construction of automorphic Green functions for special divisors on $X_K$ as regularized theta lifts of Whittaker forms.
For background on Arakelov theory we refer to \cite{SABK}, \cite{BKK}.

\subsection{Regularized theta lifts of Whittaker forms}

Let $L\subset V$ be an even $\calO_F$-lattice.
Recall from Section \ref{sect:4.3} that there is a corresponding Siegel theta function $\Theta_L(\tau,z,h)$
for $\tau\in \H^d$, $z\in \D$, and $h\in H(\hat\Q)$. It has weight $\tilde\kappa$.

For any $\mu\in L'/L$ and any totally positive $m\in  \partial^{-1}+Q(\mu)$, let $f_{m,\mu}(\tau,s)$ be the Whittaker form of weight of weight
\[
k=(\frac{2-n}{2}, \frac{2+n}{2},\dots,\frac{2+n}{2})
\]
with parameter $s$ defined in
\eqref{eq:deffms}.
The automorphic Green function $\Phi_{m,\mu}(z,h,s)$ for the divisor $Z(m,\mu)$
is defined as the regularized theta lift of  $f_{m,\mu}(\tau,s)$,
\begin{align*}
\Phi_{m,\mu}(z,h,s)&=
\frac{1}{\sqrt{D}}\int_{\tilde \Gamma_\infty\bs \H^d}^{reg}\langle f_{m,\mu}(\tau,s), \Theta_L(\tau, z,h)\rangle (v_2\cdots v_d)^{\ell/2}\,d\mu(\tau)\\
&=\frac{1}{\sqrt{D}}\int_{v\in (\R_{>0})^d}\left(\int_{u\in \calO_F\bs \R^d}
\langle f_{m,\mu}(\tau,s), \Theta_L(\tau, z,h)\rangle \,du\right) (v_2\cdots v_d)^{\ell/2}
\,\frac{dv}{\norm(v)^{2}}.
\end{align*}
The regularized integral converges if $\Re(s)>s_0+2$, where $s_0=n/2$.
It is proved in \cite{Br2} that $\Phi_{m,\mu}(z,h,s)$ has a meromorphic continuation to the whole $s$ plane with a simple pole at $s=s_0$. It satisfies a functional equation relating the values at $s$ and $-s$.

Now let $f\in  H_{k,\bar\rho_L}$ be a {\em harmonic} Whittaker form and write
\begin{align}
\label{eq:fexp}
f=\sum_{\mu\in L'/L} \sum_{m\gg 0} c(m,\mu)f_{m,\mu}(\tau).
\end{align}
The regularized theta lift $\Phi(z,h,f)$ of $f$ is defined as the constant term in the Laurent expansion at $s=s_0$ of
\[
\Phi(z,h,s,f):=\sum_{\mu\in L'/L} \sum_{m\gg 0} c(m,\mu)\Phi_{m,\mu}(z,h,s).
\]
It has a logarithmic singularity along the divisor $-2Z(f)$, where
\begin{align}
\label{eq:divf}
Z(f)=\sum_{\mu\in L'/L} \sum_{m\gg 0} c(m,\mu)Z(m,\mu).
\end{align}
In view of \cite[Corollary 5.16]{Br2}, the function $\Phi(z,h,f)$ is an Arakelov Green function for the divisor $Z(f)$.
Recall that the degree of $Z(m,\mu)$ is given by
\begin{align}
\label{eq:divf2}
B_L(m,\mu)= -\frac{\deg(Z(m,\mu))}{\vol(X_K)},
\end{align}
where $B_L(m,\mu)$ denotes the $(m,\mu)$-th coefficient
of the Eisenstein series $E_L(\tau,s_0,\kappa)$ as in Section \ref{sect:4.6} (see e.g. \cite[Remark 6.5]{Br2}).
If we put
\begin{align}
\label{eq:Bf}
B(f)&=\sum_{\mu\in L'/L} \sum_{m\gg 0} c(m,\mu)B_L(m,\mu), 
\end{align}
then by \cite[Corollary 5.9]{Br2}
the residue of $\Phi(z,h,s,f)$ at $s_0$ is equal to $-2B(f)$.  Consequently,
we have
\begin{align}
\Phi(z,h,f)= \lim_{s\to s_0} \left(\Phi(z,h,s,f)+\frac{ 2B(f)}{s-s_0} \right).
\end{align}

\subsection{A variant of the theta integral}

We now give a different regularized integral representation for $\Phi_{m,\mu}(z,h,s)$ which converges at $s_0$. It also leads to a new integral representation for the Green function $\Phi(z,h,f)$.
The idea is the same as in \cite[Section 6.2]{Br2}: We subtract the ``Eisenstein contribution'' of the Siegel theta function $\Theta_L$.
The remaining ``cuspidal contribution'' satisfies a better growth estimate as $v_i\to 0$ and therefore leads to a larger domain of convergence.
As before we
assume that $V$ is anisotropic over $F$ or that its Witt rank
is smaller than $n$.

We define the cuspidal part of the Siegel theta function by
\begin{align}
\label{eq:siecu}
\tilde \Theta_{L}(\tau,z,h)=\Theta_{L}(\tau,z,h)-E_L(\tau,s_0,\tilde\kappa).
\end{align}
It follows from Lemma \ref{Siegel-Weil} that the constant terms at all cusps of $\tilde\Gamma$ of this function vanish or are rapidly decreasing. Consequently,
$\tilde \Theta_{L}(\tau,z,h)$ is rapidly decreasing.

\begin{proposition}
\label{prop:thetatilde}
Assume the above hypothesis on  $V$.
\begin{enumerate}
\item[(i)]
The function  $\tilde \Theta_{L}(\tau,z,h)v^{\tilde\kappa/2}$ is  bounded  on $\H^d$.
\item[(ii)]
For $v_i\to 0$ we have uniformly in $u$ that $\tilde \Theta_{L}(\tau,z,h)=O(v^{-\tilde\kappa/2})$.
\item[(iii)]
The $m$-th Fourier coefficient of $\tilde \Theta_{L}(\tau,z,h)$ is bounded by
$O(v_1^{\frac{2-n}{4}})$ as $v_i\to 0$.
\end{enumerate}
\end{proposition}

\begin{proof}
Since  $\tilde \Theta_{L}(\tau,z,h)$ is rapidly decreasing at all cusps of $\tilde\Gamma$, (i) and (ii) follow by the usual argument. It remains to prove (iii). The behavior of the Fourier coefficients as $v_1\to 0$ is a direct consequence of (ii). Moreover, since $\tilde \Theta_{L}(\tau,z,h)$ is holomorphic in $\tau_2,\dots,\tau_d$ its Fourier coefficients are bounded as $v_i\to 0$ for $i=2,\dots,d$.
\end{proof}

\begin{proposition}
We have that
\begin{align*}
\Phi_{m,\mu}(z,h,s)&=
\frac{1}{\sqrt{D}}\int_{\tilde \Gamma_\infty\bs \H^d}^{reg}\langle f_{m,\mu}(\tau,s), \tilde \Theta_L(\tau, z,h)\rangle (v_2\cdots v_d)^{\ell/2}\,d\mu(\tau)\\
&\phantom{=}{}+\frac{2n}{(s^2-s_0^2)\Gamma(\frac{s-s_0}{2}+1)}\frac{\deg(Z(m,\mu))}{\vol(X_K)}.
\end{align*}
Here the regularized theta integral converges for $\Re(s)>1$.
\end{proposition}

\begin{proof}
By definition we have
\begin{align*}
\Phi_{m,\mu}(z,h,s)&=
\frac{1}{\sqrt{D}}\int_{\tilde \Gamma_\infty\bs \H^d}^{reg}\langle f_{m,\mu}(\tau,s), \tilde \Theta_L(\tau, z,h)\rangle (v_2\cdots v_d)^{\ell/2}\,d\mu(\tau)\\
&\phantom{=}{}+\frac{1}{\sqrt{D}}\int_{\tilde \Gamma_\infty\bs \H^d}^{reg}\langle f_{m,\mu}(\tau,s), E_L(\tau,s_0, \tilde\kappa)\rangle (v_2\cdots v_d)^{\ell/2}\,d\mu(\tau).
\end{align*}
In view of Lemma \ref{Siegel-Weil}, the second summand on the right hand side
is equal to
\begin{align*}
&\frac{1}{\sqrt{D}\vol(X_K)}\int_{X_K}
\int_{\tilde \Gamma_\infty\bs \H^d}^{reg}\langle f_{m,\mu}(\tau,s), \Theta_L(\tau, z,h)\rangle (v_2\cdots v_d)^{\ell/2}\,d\mu(\tau)\,\Omega^n\\
&=\frac{1}{\vol(X_K)}\int_{X_K} \Phi_{m,\mu}(z,h,s)\,\Omega^n \\
&=\frac{2n}{(s^2-s_0^2)\Gamma(\frac{s-s_0}{2}+1)}\frac{\deg(Z(m,\mu))}{\vol(X_K)}.
\end{align*}
Here the last equality follows from Theorem 5.7 of \cite{Br2} applied with the test function $1$. This proves the identity of the proposition. The convergence statement follows from
Proposition \ref{prop:thetatilde} and the asymptotics for Whittaker forms, see Section \ref{sect:3.2} and \cite[Section~4]{Br2}.
\end{proof}

Let $f\in  H_{k,\bar\rho_L}$ be a harmonic Whittaker form.
If $n>2$, we define
\begin{align}
\label{eq:tildephi}
\tilde\Phi(z,h,f)&=
\frac{1}{\sqrt{D}}\int_{\tilde \Gamma_\infty\bs \H^d}^{reg}\langle f(\tau), \tilde \Theta_L(\tau, z,h)\rangle (v_2\cdots v_d)^{\ell/2}\,d\mu(\tau)
\\ \nonumber
&=\frac{1}{\sqrt{D}}\int_{v\in (\R_{>0})^d}\left(\int_{u\in \calO_F\bs \R^d}
\langle f(\tau), \tilde\Theta_L(\tau, z,h)\rangle \,du\right) (v_2\cdots v_d)^{\ell/2}
\,\frac{dv}{\norm(v)^{2}}.
\end{align}
Note that the regularized integral converges.
If $n\geq 1$, we define $\tilde\Phi(z,h,f)$ as the value at $s'=0$ of the holomorphic continuation in $s'$ of
\begin{align}
\label{eq:tildephi2}
\frac{1}{\sqrt{D}}\int_{\tilde \Gamma_\infty\bs \H^d}^{reg}\langle f(\tau), \tilde \Theta_L(\tau, z,h)\rangle (v_2\cdots v_d)^{\ell/2}\norm(v)^{s'}\,d\mu(\tau).
\end{align}
Here the regularized integral converges for $\frac{2-n}{4}<\Re(s')<1$. The fact that it has a continuation in $s'$ can be deduced from the continuation in $s$ of $\Phi(z,h,s,f)$.

\begin{corollary}
Let $f\in  H_{k,\bar\rho_L}$ be a harmonic Whittaker form. We have
\begin{align*}
\Phi(z,h,f)&=
\tilde \Phi(z,h,f)+B(f)\big(\Gamma'(1)+2/n\big).
\end{align*}
\end{corollary}

\begin{remark}
By construction, $\tilde \Phi(z,h,f)$ is normalized such that
\[
\int_{X_K}\tilde \Phi(z,h,f) \Omega^n =0.
\]
\end{remark}

\section{CM values of automorphic Green functions}

\label{sect:6}


Let $W\subset V$ be a maximal totally positive definite subspace.
Recall that the CM cycle $Z(W)$ is given by
\begin{align*}
H_W(\Q)\bs \D_W\times H_W(\hat\Q)/(H_W(\hat \Q)\cap K)
\longrightarrow X_K.
\end{align*}
If $f\in   H_{k,\bar\rho_L}$ is a harmonic Whittaker form, we aim to compute the CM value
\begin{align}
\label{tocompute}
\Phi(Z(W),f):=
\sum_{[z,h]\in  Z(W) } \Phi(z,h,f):= \sum_{ [z, h] \in \supp(Z(W))} \frac{2}{w_K} \Phi(z, h, f).
\end{align}
Recall that each  point in $Z(W)$ is counted with  multiplicity $2/w_K$ (see Section \ref{sect:4.3}).

By means of the splitting $V=W\oplus W^\perp$, we obtain $\calO_F$-lattices
\begin{align*}
P&=L\cap W,\\
N&=L\cap W^\perp.
\end{align*}
The lattice $P$ is totally positive definite, while $N$ has signature $((0,2),(2,0),\dots ,(2,0))$.
Then  $N\oplus P\subset L$ is a sublattice of finite index.
As in \cite[Lemma 3.1]{BY}, we may view $f$ as a Whittaker form $f_{P\oplus N}$ for the sublattice $P\oplus N$.
We have
\begin{align}
\label{eq:latticeflex}
\langle f, \Theta_L \rangle =\langle f_{P\oplus N}, \Theta_{P \oplus N} \rangle,
\end{align}
so we may assume below that $L=P\oplus N$ if we
replace $f$ by $f_{P\oplus N}$.

For any cusp form $g\in S_{\kappa,\rho_L}$ we define an $L$-function by means
of the convolution integral
\begin{align}
\label{eq:L} L(g, W,s)=\big( \Theta_P(\tau)\otimes E_N(\tau,s,\kappa_N),\,
g(\tau)\big)_{Pet},
\end{align}
where $E_N(\tau,s,\kappa_N)$ denotes the incoherent Eisenstein series of weight $\kappa_N=(1,\dots, 1)$
associated to the lattice $N\subset W^\perp$ defined in Section \ref{sect:4.4}.
The Petersson scalar product is normalized as in \cite[(4.21)]{Br2}.
The meromorphic continuation of the Eisenstein series
$E_N(\tau,s,\kappa_N)$ leads to a meromorphic continuation of $L(g, W,
s)$ to the whole complex plane. At $s=0$, the center of symmetry,
$L(g, W, s)$ vanishes because the Eisenstein series vanishes at that point by Proposition \ref{prop4.9}.
%
Let
\begin{align}
g(\tau)&=\sum_{\mu\in L'/L} \sum_{m\gg0} b(m,\mu) q^m\chi_\mu,\\
\Theta_P(\tau)&=\sum_{\mu\in P'/P} \sum_{m } r(m,\mu)
q^m\chi_\mu
\end{align}
be the Fourier expansion of $g$ and $\Theta_P$, respectively. Using
the usual unfolding argument, one obtains a Dirichlet series
expansion of $L(g,W,s)$. For instance, if the narrow class number of $F$ is one, we have
\begin{align}
\label{eq:luf} L(g, W,
s)=(4\pi)^{-d(s+n)/2}\Gamma\left(\tfrac{s+n}{2}\right)^d\sum_{m\in \partial^{-1}/(\calO_F^\times)^2}\sum_{\mu\in
P'/P} r(m,\mu) \overline{b(m,\mu)}  \norm(m)^{-(s+n)/2}.
\end{align}
Here the first sum runs through the orbits of $\partial^{-1}$ modulo the action of $(\calO_F^\times)^2$, the group of units which are a square.


Before stating our main result we need the following lemmas.

\begin{lemma}
\label{lem:gr1}
Let $m\in F$ be totally positive. The $m$-th Fourier coefficient of $\Theta_P\otimes \calE_N -2\calE_L$ is $O( v^{-\kappa/2})$ as $v_i\to 0$.
\end{lemma}

\begin{proof}
Let $\calE_{N,0}(\tau)$  denote the $S_N$-valued constant term of the Fourier expansion of $\calE_{N}(\tau)$ at the cusp $\infty$.
Let $\calE_{L,0}(\tau)$  denote the $S_L$-valued constant term of the Fourier expansion of $\calE_{L}(\tau)$ at the cusp $\infty$.
Then $\calE_N-\calE_{N,0}$ and $\calE_L-\calE_{L,0}$ are rapidly decreasing at the cusp $\infty$.
Consequently,
\begin{align*}
\Theta_P\otimes \calE_N -2\calE_L &= \Theta_P\otimes \calE_{N,0} -2\calE_{L,0} +R_1\\
&=\chi_{0+P}\otimes\calE_{N,0} -2\calE_{L,0} +R_2,
\end{align*}
where $R_1$ and $R_2$ are rapidly decreasing at the cusp $\infty$.
It follows from Proposition \ref{prop4.7} and Proposition \ref{prop4.9}
that
the $\log(\norm(v))$ parts of
$\chi_{0+P}\otimes \calE_{N,0}$ and $2\calE_{L,0}$ cancel.
Hence
\[
\chi_{0+P}\otimes\calE_{N,0} -2\calE_{L,0}=\beta + \alpha \norm(v)^{-s_0}
\]
for some $S_L$-valued constants $\beta$ and $\alpha$. Moreover, $\alpha=0$ when $n=1$.
Consequently, near the cusp $\infty$, the function
$\Theta_P\otimes \calE_N -2\calE_L$ is the sum of a  constant and a function that decays at least as $O(\alpha \norm(v)^{-s_0})$.
The same analysis applies to the other cusps.
Hence there exists a linear combination $G$ of holomorphic Eisenstein series of weight $\kappa$ for the group $\tilde\Gamma$ such that
$\Theta_P\otimes \calE_N -2\calE_L-G$ decays as  $O(\alpha \norm(v)^{-s_0})$ at all cusps. (When $\alpha=0$ it is rapidly decreasing.)
Therefore
$|\Theta_P\otimes \calE_N -2\calE_L-G|v^{\kappa/2}$ is bounded on $\H^d$.
This implies the assertion.
\end{proof}

\begin{lemma}
\label{lem:gr2}
Let $m\in F$ be totally positive, and let $j\in \{1,\dots,d\}$.
The $m$-th Fourier coefficient of $\Theta_P\otimes \calE_N^{(j)} -2\calE_L^{(j)}$ is $O(1)$, as $v_i\to 0$ for $i\neq j$. It is $O( v^{-\kappa_j/2})$ as $v_j\to 0$.
\end{lemma}

\begin{proof}
The function $\Theta_P\otimes \calE_N^{(j)} -2\calE_L^{(j)}$ is holomorphic in $\tau_i$ for $i\neq j$.
Hence its $m$-th Fourier coefficient is independent of $\tau_i$, and therefore bounded as $v_i\to 0$.

Moreover, we have
\[
\Theta_P\otimes \calE_N^{(j)} -2\calE_L^{(j)}=
(\Theta_P\otimes \calE_N -2\calE_L)
- \left(\Theta_P\otimes  (\calE_N-\calE_N^{(j)}) -2(\calE_L -\calE_L^{(j)})
\right).
 \]
The second quantity  on the right hand side is holomorphic in $\tau_j$, and therefore its $m$-th Fourier coefficient is bounded as $v_j\to 0$. Hence the growth in $v_j$ follows from Lemma~\ref{lem:gr1}.
\end{proof}

\begin{theorem}
\label{thm:fund}
Let $f\in   H_{k,\bar\rho_L}$ and assume that
$Z(W)$ and $Z(f)$ do not intersect on $X_K$.
If $n>2$ then the the value of the automorphic Green function
$\tilde\Phi(z,h,f)$ at the CM cycle $Z(W)$ is given by
\begin{align*}
\tilde\Phi(Z(W),f)&=-
\frac{\deg(Z(W))}{\sqrt{D}}\cdot
\int_{\tilde \Gamma_\infty\bs \H^d}^{reg}
\big\langle \overline{\delta ( f)}, \,\Theta_P\otimes \calE_N^{(1)}-  2\calE_L^{(1)}\big\rangle v^\kappa d\mu(\tau).
\end{align*}
If $n\geq 1$ then the CM value $\tilde\Phi(Z(W),f)$ is given by the value at $s'=0$ of the holomorphic continuation in $s'$ of
\begin{align*}
-\frac{\deg(Z(W))}{\sqrt{D}}\cdot
\int_{\tilde \Gamma_\infty\bs \H^d}^{reg}
\big\langle \overline{\delta ( f)}, \,\Theta_P\otimes \calE_N^{(1)}-  2\calE_L^{(1)}\big\rangle  v^\kappa \norm(v)^{s'}d\mu(\tau).
\end{align*}
Here the regularized integral converges for $\frac{2-n}{4}<\Re(s')<1$.
\end{theorem}

\begin{proof}
First, we assume that $n>2$ so that we can use the integral representation \eqref{eq:tildephi}.
In view of \eqref{eq:latticeflex} we may also assume without loss of generality that $L=P\oplus N$.
By definition we have
\begin{align}
\label{eq:tpc}
\tilde\Phi(Z(W),f)&=
 \frac{1}{\sqrt{D}}\cdot
\int_{\tilde \Gamma_\infty\bs \H^d}^{reg}
\sum_{[z,h]\in Z(W)}
\big\langle f(\tau), \tilde \Theta_L(\tau, z_W^+,h)\big\rangle (v_2\cdots v_d)^{\ell/2}\, d\mu(\tau).
\end{align}

For $z=z_W^\pm$ and $h\in
H_W(\hat \Q)$, the Siegel theta function $\Theta_L(\tau,z,h)$ splits up
as a product
\begin{align}
\label{splittheta} \Theta_L(\tau,z_W^\pm,h)= \Theta_P(\tau)\otimes
\Theta_N(\tau,z_W^\pm,h).
\end{align}
Here $\Theta_P(\tau)=\Theta_P(\tau,1)$ is the holomorphic
$S_P$-valued theta function of parallel weight $n/2$ associated to the
totally positive definite lattice $P$.
On the other hand, according to the  Siegel Weil formula \eqref{eq:kms2} we have
\begin{align}
\sum_{[z,h]\in  Z(W)} \Theta_{L}(\tau,z,h)
=\frac{\deg(Z(W))}{2}\cdot E_N(\tau,s,\tilde\kappa_N),
\end{align}
where $E_N(\tau,s,\tilde\kappa_N)$ denotes the coherent Hilbert Eisenstein series defined in Section \ref{sect:4.3} of weight $\tilde\kappa_N=(-1,1,\dots,1)$  associated to the lattice $N\subset W^\perp$.
Inserting this into \eqref{eq:tpc} we obtain
\begin{align*}
\tilde\Phi(Z(W),f)&=
\frac{\deg( Z(W))}{2\sqrt{D}}\\
&\times\int_{\tilde \Gamma_\infty\bs \H^d}^{reg}
\left\langle f(\tau), \Theta_P(\tau)\otimes E_N(\tau,0,\tilde\kappa_N)
-2E_L(\tau,\tfrac{n}{2},\tilde\kappa_V)
\right\rangle (v_2\cdots v_d)^{\ell/2}\, d\mu(\tau).
\end{align*}
In view of \eqref{eq:eisrel2}, if we write
$\eta= (v_2\cdots v_d)^{\ell/2}d\tau_1 d\mu(\tau_2)\cdots d\mu(\tau_d)$,
we have the following identities of differential forms on $\H^d$:
\begin{align*}
-2\bar \partial \calE_N^{(1)}(\tau)\eta &=-2\bar \partial_1 \calE_N(\tau)\eta=E_N(\tau,0,\tilde\kappa_N)(v_2\cdots v_d)^{\ell/2}d\mu(\tau),\\
-2\bar \partial \calE_L^{(1)}(\tau)\eta &=-2\bar \partial_1 \calE_L(\tau)\eta=E_L(\tau,n/2,\tilde\kappa_V)(v_2\cdots v_d)^{\ell/2}d\mu(\tau).
\end{align*}
Consequently, we find
\begin{align*}
\tilde\Phi(Z(W),f)&=-
 \frac{\deg( Z(W))}{\sqrt{D}}\int_{\tilde \Gamma_\infty\bs \H^d}^{reg}
\left\langle f(\tau), \bar\partial \big( \Theta_P(\tau)\otimes \calE_N^{(1)}(\tau)
-  2\calE_L^{(1)}(\tau)\big)\eta
\right\rangle \\
&=-
 \frac{\deg( Z(W))}{\sqrt{D}}\int_{\tilde \Gamma_\infty\bs \H^d}^{reg}
d \left\langle f, \Theta_P\otimes \calE_N^{(1)}\eta
-  2\calE_L^{(1)}\eta
\right\rangle \\
&\phantom{=}{}-
 \frac{\deg( Z(W))}{\sqrt{D}}\int_{\tilde \Gamma_\infty\bs \H^d}^{reg}
\big\langle \overline{\delta ( f)}, \Theta_P\otimes \calE_N^{(1)}
-  2\calE_L^{(1)}(\tau)
\big\rangle v^\kappa d\mu(\tau),
\end{align*}
where $\delta $ is the differential operator on harmonic Whittaker forms defined in \eqref{def:delta}.
We have used that $\bar\partial(f\eta)= -\overline{\delta (f)}v^\kappa d\mu(\tau)$, see \cite[(6.2)]{Br2}.

We now show that the quantity
\begin{align}
\label{eq:q2}
 \frac{1}{\sqrt{D}}\int_{\tilde \Gamma_\infty\bs \H^d}^{reg}
d \left\langle f, \Theta_P\otimes \calE_N^{(1)}\eta
-  2\calE_L^{(1)}\eta
\right\rangle
\end{align}
vanishes. For $T>0$ we let $R_T\subset \R_{>0}^d$ be the rectangle
\[
R_T= [1/T,T]\times\dots \times [1/T,T].
\]
Using the invariance of the integrand under translations,
we find by  Stokes' theorem that \eqref{eq:q2} is equal to
\begin{align*}
\frac{1}{\sqrt{D}}\lim_{T\to \infty} \int_{\partial R_T}\int_{\calO_F\bs \R^d}
\left\langle f, \Theta_P\otimes \calE_N^{(1)}
-  2\calE_L^{(1)}
\right\rangle\eta.
\end{align*}
Notice that only the parts of the boundary $\partial R_T$ where $v_1=1/T$ or $v_1=T$ give a non-zero contribution.
Carrying out the integration over $u$, we see that \eqref{eq:q2}
is equal to
\begin{align}
\nonumber
&\lim_{v_1\to 0} \int_{v_2,\dots,v_d=0}^\infty
g_0(v)(v_2\cdots v_d)^{\ell/2-2}\, dv_2\dots dv_d\\
\label{eq:q3}
&-\lim_{v_1\to \infty} \int_{v_2,\dots,v_d=0}^\infty
g_0(v)(v_2\cdots v_d)^{\ell/2-2}\, dv_2\dots dv_d,
\end{align}
where $g_0(v)$ denotes the constant term of the Fourier series
\[
g(\tau)=\left\langle f, \Theta_P\otimes \calE_N^{(1)}
-  2\calE_L^{(1)}
\right\rangle.
\]
Using  Lemma \ref{lem:gr1} and the asymptotic behavior \eqref{Masy1} for harmonic Whittaker forms,
we see that
\[
g_0(\tau)=O(v_1^{(n-2)/4} )
\]
as $v_i\to 0$. Taking into account our assumption that $n>2$, this implies that
the first summand in \eqref{eq:q3}
vanishes. Moreover, it follows from the Fourier expansions of $\calE_N^{(1)}$ and $\calE_L^{(1)}$ that
\[
\int_{v_2,\dots,v_d=0}^\infty
g_0(v)(v_2\cdots v_d)^{\ell/2-2}\, dv_2\dots dv_d = O(v_1^{-1})
\]
as $v_1\to \infty$. In fact, for the contribution from $\calE_L^{(1)}$ this follows from the decay of the function $J(a,t)$ occurring in the Fourier coefficients. We have $J(a,t)=O(t^{-1})$ as $t\to \infty$ for $a>0$.
The contribution from $\Theta_P\otimes \calE_N^{(1)}$ is actually exponentially decreasing due to the exponential decay of the incomplete gamma function occurring in the Fourier coefficients of $\calE_N^{(1)}$.
Hence the second summand in \eqref{eq:q3} valishes as well.
This concludes  the proof of the theorem when $n>2$.

Finally, we note that for $n\leq 2$, we may use the integral representation  \eqref{eq:tildephi2}, and argue similarly with the appropriate modifications taking into account the extra $\norm(v)^{s'}$ term.
\end{proof}

\section{Shimura varieties associated to incoherent quadratic spaces} \label{sect:global}

\label{sect:7}

Actually, the complex manifold $X_K=H(\Q) \backslash \mathbb D \times H(\hat{\Q})/K$  studied in the previous sections is just one complex piece of a Shimura variety
$\mathbb X_K$ over $F$: We have $X_K =\mathbb X_K \times_{F, \sigma_1} \mathbb C$. To study the arithmetic of $\mathbb X_K$, we need also the complex points of $\mathbb X_K$
with respect to the other embeddings $\sigma_j: F \to \mathbb C$ and its integral structures at the finite primes  of $F$.
They are associated to `companion quadratic spaces' of $V$ over $F$. A more convenient way to describe these quadratic spaces is to use the notion of incoherent quadratic spaces first described by Kudla \cite{Ku:Annals}.

A quadratic space over $\A_F$ is a free $\A_F$-module $\V$ of
finite rank together with a non-degenerate quadratic form $Q:
\V\to \A_F$. If $w$ is a place of $F$, then $\V_w=\V\otimes
_{\A_F} F_w$ together with the induced quadratic form $Q_w:\V_w\to
F_w$ is a quadratic space over the local field $F_w$. We may view
$\V$ as the restricted product
\[
\V=\prod_{w\leq \infty} \V_w.
\]
The determinant of $\V$ defines an element of the idele group
$\A_F^\times$. Let $\ell$ denote the rank of $\V$, and let
$\chi_{\V,w}=((-1)^{\frac{\ell(\ell-1)}{2}}\det \V_w, \cdot )$ be the
quadratic character of $\V_w$ given by the Hilbert symbol. Then
$\chi_\V=\prod_w \chi_{\V,w}$ is a well defined character
$\A_F^\times\to \{\pm 1\}$. At a finite place $w < \infty$ the local
quadratic space $\V_w$ is determined up to isometry by $\ell$, the
character $\chi_{\V,w}$, and the Hasse invariant $\Hasse(\V_w)$.
At an infinite place $w$, the local space is determined by the
signature. We have $\Hasse(\V_w)=1$ for all but finitely many
places. We define
$$
\Hasse(\V)=\prod_{w\leq \infty}\Hasse(\V_w)\in \{\pm 1\}.
$$

A quadratic space $\V$ over $\A_F$ is called {\it admissible} if
$\chi_\V$ is trivial on $F^\times$. Such a
space is called coherent if there exists a global quadratic space
$V$ over $F$ such that $V(\A_F) =V \otimes_F \A_F \cong \V$. In this
case $V$ is uniquely determined up to isometry, and we will
identify $\V$ with $V(\A_F)$. If there exists no such global space,
$\V$ is called incoherent. An admissible quadratic space $\V$ is coherent if
and only if $\Hasse(\V)=1$.

Let $(\V, Q)$ be an admissible incoherent quadratic space over
$\A_F$ which is totally positive definite of
dimension $\ell=n+2$.
For every real embedding $\sigma_j$ there is an up to isometry
unique
quadratic space $V_j$ over $F$, called the
neighboring quadratic
space of $\V$ at $\sigma_j$ (following Kudla), such that
\begin{align}
\label{eq:nb}
V_{j, w}=
V_j \otimes_F F_w \cong
\begin{cases}\V_w,& \text{ for $w \ne \sigma_j$,}\\
\R^{n,2},& \text{ for $w =\sigma_j$.}
\end{cases}
\end{align}
Here $\R^{n,2}$ denotes the standard quadratic space over $\R$ of
signature $(n,2)$, that is $\R^{n+2}$ with the quadratic form
$x\mapsto x_1^2+\dots+x_n^2-x_{n+1}^2-x_{n+2}^2$. We fix the
neighboring spaces $V_j$, and we also fix isometries
\begin{align}
\label{eq:nu}
\nu_j: \hat{\V} \longrightarrow \hat{V}_j,
\end{align}
where $\hat\V=\prod_{\mathfrak p < \infty}\V_\mathfrak p$
and
$\hat{V}_j=V_j\otimes_F \hat F$.

We let
$H_j=\Res_{F/\Q}\GSpin(V_j)$ be the algebraic group over
$\Q$ given by Weil restriction of scalars.
We also consider the restricted direct product
\[
\calH=\GSpin(\hat{\V})=\prod_{\mathfrak p < \infty}\GSpin(\V_\mathfrak p).
\]
Via the isomorphism \eqref{eq:nu}
we will
identify $\calH$ with $\Gspin(V_j) (\hat{\mathbb Q})$.


Let $\D=\D^{n,2}$ be the Hermitian symmetric domain of oriented negative definite
$2$-planes in
$\R^{n,2}$. It has two components corresponding to the two possible choices of the orientation.
The group
$\Gspin(V_j)$ acts on $\mathbb D$ through the isomorphism
$V_{j, \sigma_j} \cong \R^{n,2}$.
Let $K\subset \calH$ be a compact
open subgroup. The Shimura variety
\[
X_{K,j}=H_j(\mathbb Q) \backslash \mathbb D \times
H_j(\hat{\mathbb Q})/K =H_j(\mathbb Q) \backslash \mathbb D \times
\calH/K
\]
associated to $(V_j,K)$ has a canonical model
$\X_{K, j}$ defined over $\sigma_j(F)\subset
\mathbb C$, see \cite{Shih}.

\begin{lemma}
\label{lem:shim:conj}
Let the notation be as above. There is a quasi-projective
variety $\X_K$ defined over $F$ such that for each embedding $\sigma_j: F
\to \mathbb R\subset \C$, the base change $\X_K \times_{F, \sigma_j} \sigma_j(F)$
is equal to $\X_{K,j}$.
We call $\X_K$ the Shimura variety associated to $(\V,K)$.
\end{lemma}

\begin{proof}
By the theory of conjugates of Shimura varieties
\cite[Section 4]{Milne}, the base change
$\sigma_j \sigma_1^{-1} \X_{K, 1}$ is equal to $\X_{K, j}$. So the variety $\X_K= \sigma_1^{-1}\X_{K, 1}$
is the desired Shimura variety associated to $(\V, K)$.
\end{proof}

If we view $\X_K$ as a scheme over $\Q$ via $\X_K\to \Spec(F)\to \Spec(\Q)$, we have that
$\X_K\times_\Q \C =\coprod_j X_{K,j}$.
Note that we may take for the quadratic space $V$ of the previous sections the space $V_1$ of the present section. Then we may identify the Shimura variety $X_K$ of the previous sections with the component $X_{K,1}$ in the setup of the present section. In the context of arithmetic intersection theory (see e.g. \cite{SABK}, \cite{BKK}) it is important to consider all Galois conjugates $X_{K,j}$ of $X_{K,1}$ simultaneously.


\subsection{Special cycles}

\label{sect:7.1}

A {\it coherent (positive definite) subspace} of $\V$ is defined to be a tuple
$W=(W, \iota, (\iota_j))$ where $W$ is  a totally
positive definite quadratic space over $F$ of dimension $r \le n$ together
with embeddings
\begin{align}
\label{eq:wdata}
\iota: W(\A_F) \longrightarrow \V, \quad \iota_j:  W \longrightarrow
V_j,
\end{align}
which are compatible with the isomorphisms $\nu_j$, that is, such that
\[
\xymatrix{ \hat W \ar[r]^\iota \ar[rd]_{\iota_j} & \hat \V\ar[d]^{\nu_j}\\
&\hat V_j}
\]
commutes for every $j=1,\dots,d$.
Let $\W=\iota(W(\A_F))\subset \V$ and write $\mathbb
W^\perp$ for its orthogonal complement in $\V$. We denote by  $\calH_W$ the pointwise stabilizer of
$W$ in $\calH$. Analogously, let  $W_j^\perp$ be the orthogonal complement of $W$ in
$V_j$, and denote by $H_{W, j}$ the pointwise stabilizer of $W_j$ in $H_j$.
Then $H_{W, j}\cong \Res_{F/\Q}\GSpin(W_j^\perp)$, and
$H_{W,j}(\hat\Q)$ is isomorphic to $\calH_W$ via the isomorphism induced by $\nu_j$.
We write $\D_{W,j}$ for the sub-manifold of $\D$ of given by the oriented negative definite
$2$-planes in $\R^{n,2}$ which are orthogonal to $W_{j,\sigma_j}\subset V_{j,\sigma_j}\cong\R^{n,2}$.

For any $h \in \calH\cong H_j(\hat\Q)$
let $Z_j(W,h)=Z_j(W,\nu_j(h))$ be the special cycle on $X_{K,j}$ defined in Section \ref{sect:2} (with respect to $V=V_j$). It is a cycle of codimension $r$ defined over
$\sigma_j(F)\subset\C$.
%
Analogously to Lemma \ref{lem:shim:conj}, there is an algebraic cycle $Z(W,h)$ on $\X_K$
defined over $F$ whose image under the base change to $X_{K,j}$ via $\sigma_j$ is equal to $Z_j(m,\mu)$ for all $j$.

In the present paper we are interested in two particular cases of
this construction. First, if $r=n$, then $\W_\infty^\perp$ is totally positive
definite  of dimension $2$,
and $W_{j}^\perp $ has signature $(0, 2)$ at the place $\sigma_j$ and
signature $(2, 0)$ at all other infinite places. In this case,
$\mathbb D_{W,j}$ consists of two  points $z_{W,j}^\pm$, which are
$W_{j, \sigma_j}^\perp$ with the two possible choices of an orientation.
%
The group $\GSpin(\W^\perp)$ can be be identified with $\A_\kay^\times$ for a totally
imaginary quadratic extension $\kay $ of $F$. The
corresponding dimension zero cycle $Z(W)=Z(W,1)$ is called the CM
cycle associated to $W$.
According to \eqref{eq:degCM1},  $\deg(Z_j(W))$ is independent of $j$, and is equal to $\frac{4}{w_{K_W}} |\kay^\times\backslash \hat{\kay}^\times/K_W|$
with $K_W=\hat{\kay}^\times\cap K$.
We define $\deg(Z(W)):=\deg(Z_j(W))$ as the degree of $Z(W)$.

The second case we are interested in is $r=1$.
For a
totally
positive definite quadratic space $W$ over $F$ of dimension $1$ together with compatible embeddings,
and for $h\in \calH$, we have a divisor $Z(W, h)$ on $\X_K$.
We consider certain sums of these divisors, called weighted divisors.
They generalize Heegner divisors on modular curves.
Let $m\in F$ be totally positive, and let
$\varphi\in S(\hat\V)^K$ be a $K$-invariant Schwartz function.
First, assume that there is a totally
positive definite quadratic space $W$ over $F$ of dimension $1$ together with compatible embeddings
$\iota$, $(\iota_j)$ as in \eqref{eq:wdata} that represents $m$. Then let $x_0\in W$
 with $Q(x_0)=m$ and define
\begin{align*}
Z(m,\varphi)= \sum_{h\in \calH_W\bs  \calH/K } \varphi(h^{-1} x_0) Z(Fx_0,h).
\end{align*}
The sum is finite, and $Z(m,\varphi)$ is a divisor on $\X_K$ with complex coefficients.
If there is no such space $W$ representing $m$, we put $Z(m,\varphi)=0$.
We write $Z_j(m,\varphi)$ for the divisor on $X_{K,j}$ obtained from $Z(m,\varphi)$ by base change via $\sigma_j$. This is a special divisor as in Section~\ref{sect:2} associated to $V_j$ and $\varphi$.

Let  $\LL$ be a lattice in $\hat \V$, that is,  a free $\hat\calO_F$-submodule such that $\LL\otimes \hat F=\hat \V$. Assume that
$K$ fixes $\LL$ and acts trivially on $ \LL'/ \LL$, where $\LL'$ denotes the dual lattice. For $\mu \in  \LL'/ \LL$, we let $\chi_\mu=\Char(\mu+\LL) \in S(\hat\V)^K$ be  the characteristic function. We briefly write
\begin{align*}
Z(m,\mu)&:=Z(m,\chi_\mu),
\end{align*}
Then its complex component $Z_j(m, \mu)$ with respect to $\sigma_j$ is the special divisor defined in Section \ref{sect:2} with respect to the quadratic space $V_j$.

\subsection{Automorphic Green functions for the divisors $Z_j(m,\mu)$}

Here we briefly describe how the construction of automorphic Green functions of Section \ref{sect:green} can be adapted in order to obtain
Green functions for the divisors $Z_j(m,\mu)$.
Let $L_j\subset V_j$ be the lattice given by $\nu_j(\LL)\cap V_j(F)$.
We write $\tilde\Theta_{L_j}(\tau,z,h)$ for the corresponding Siegel theta function as in \eqref{eq:siecu}.

Let $k(j)$ be the $d$-tuple whose $j$-th component
is $\frac{2-n}{2}$ and whose $i$-th component is $\frac{n+2}{2}$ for all $i\neq j$.
For $\mu\in \LL'/\LL$ and $m\in \partial^{-1}+Q(\mu)$ totally positive,
we have a harmonic Whittaker form $f_{m,\mu}^{(j)}$ of weight $k(j)$ which is defined as in \eqref{eq:fs0} but with the $M$-Whittaker function at the $j$-th place instead of the first, that is,
\begin{align*}
f^{(j)}_{m,\mu}(\tau)
=\frac{\norm(4\pi m)^{s_0}}{(4\pi m_j)^{s_0}\Gamma(s_0)^d}
\big(\Gamma(s_0)-\Gamma(s_0,4\pi m_j v_j)\big)e^{4\pi m_j v_j}e(-\tr(m\bar\tau))\chi_\mu.
\end{align*}
Let $\tilde \Phi_{m,\mu}^{(j)}(z,h)$ be the regularized theta integral
\begin{align*}
\tilde\Phi_{m,\mu}^{(j)}(z,h)&=
\frac{1}{\sqrt{D}}\int_{\tilde \Gamma_\infty\bs \H^d}^{reg}\langle f_{m,\mu}^{(j)}(\tau), \tilde \Theta_{L_j}(\tau, z,h)\rangle \frac{\norm(v)^{\ell/2}}{v_j^{\ell/2}}\,d\mu(\tau)
\end{align*}
of $f_{m,\mu}^{(j)}$ analogously to  \eqref{eq:tildephi}. It is a Green function for the divisor $Z_j(m,\mu)$ on $X_{K,j}$.

\subsection{Archimedian height pairings and CM values}

A {\em principal part polynomial} is a Fourier polynomial of the the Form
\begin{align}
\label{eq:pppol}
\calP= \sum_{\mu\in \LL'/\LL} \sum_{\substack{m\in \partial^{-1}+Q(\mu)\\ m\gg 0}}  c(m,\mu) q^{-m}\chi_\mu.
\end{align}
For $j=1,\dots,d$ there are harmonic Whittaker forms
\[
f_\calP^{(j)}(\tau)=\sum_{\mu\in \LL'/\LL} \sum_{\substack{ m\gg 0}}  c(m,\mu) f^{(j)}_{m,\mu}(\tau)\chi_\mu
\]
of weight $k(j)$ (for $\rho_{\bar\LL}$ and $\tilde\Gamma$) associated to $\calP$.
We put $\xi(\calP)=\xi(f_\calP^{(j)})$, where $\xi(f_\calP^{(j)})$ is defined as in \eqref{def:xi}. This is a cusp form in $S_{\kappa,\rho_\LL}$ which is is independent of the choice of $j$. The principal part $\calP$ is called weakly holomorphic if $\xi(\calP)=0$.
Moreover, we consider the  divisor
\[
Z(\calP)=\sum_{\mu\in \LL'/\LL} \sum_{\substack{ m\gg 0}}  c(m,\mu) Z(m,\mu)
\]
on $\X_K$ corresponding to $\calP$. We let $\tilde\Phi(z,h,\calP)$ be the unique Green function for $Z(\calP)(\C)$ on $\X_K(\C)=\coprod_{j}X_{K,j}$ whose restriction to the component $X_{K,j}$ is equal to
the Green function
\[
\tilde\Phi^{(j)}(z,h,f_\calP^{(j)})=\sum_{\mu\in \LL'/\LL} \sum_{\substack{ m\gg 0}}  c(m,\mu) \tilde \Phi_{m,\mu}^{(j)}(z,h).
\]

Let $W=(W, \iota, (\iota_j))$ be a totally
positive definite coherent subspace of $\V$ of dimension $n$.
It defines a CM cycle $Z(W)$ on $\X_K$.
We now compute the value of $\tilde\Phi(z,h,\calP)$ at $Z(W)$.
We put
\begin{align*}
P&=\iota^{-1}(\LL)\cap W,\\
\N&=\W^\perp\cap\LL.
\end{align*}
Then $P$ is a totally positive definite lattice of dimension $n$ and $\N$ is a $2$ dimensional $\hat\calO_F$ sublattice of $\LL$.
Associated to $P$ we have a theta series $\Theta_P$, and associated to $\N$ and $\LL$ we have incoherent Eisenstein series as in Section \ref{sect:Eisenstein}. Notice that
$E_\LL(\tau, s, \kappa)=E_{L_j}(\tau, s, \kappa)$ is independent of $j$.

\begin{theorem}
\label{thm:fundtr}
Assume the above notation. Let $\calP$ be a principal part polynomial and
assume that $Z(W)$ and $Z(\calP)$ do not intersect on $\X_K(\C)$.
The value of the automorphic Green function
$\tilde\Phi(z,h,\calP)$ at the CM cycle $Z(W)$ is given by
\begin{align*}
\tilde\Phi(Z(W),\calP)
&=\deg(Z(W))\left( \CT\langle \calP, \Theta_P\otimes \calE_\N^{(0)}\rangle
-2\CT\langle \calP, \calE_\LL^{(0)}\rangle -  L'(\xi (\calP), W,0)\right).
\end{align*}
Here
$L'(\xi (\calP),W,s)$ denotes the derivative  with respect
to $s$ of the $L$-series \eqref{eq:L} associated to the cusp form
$\xi (\calP)$.
Moreover, $\CT(\cdot)$ denotes the constant term of a holomorphic Fourier series.
\end{theorem}

\begin{remark}
Note that there is a sign error in \cite[Theorem 4.7]{BY}. It should read ``$-L'(\xi(f),U,0)$''. In the proof the sign in line 3 on page 655 is wrong. The same sign error occurs in the statement of Conjecture 5.2, Conjecture 1.1, and Theorem 1.2.
\end{remark}

\begin{proof}[Proof of Theorem \ref{thm:fundtr}.]
Here we give a detailed proof in the case when $n>2$.
The case $n=1,2$ is treated analogously taking into account the additional $\norm(v)^{s'}$ term in the integral
representation of Theorem \ref{thm:fund}.

By definition we have
\[
\tilde\Phi(Z(W),\calP)
=\sum _{j=1}^d \tilde\Phi^{(j)}(Z_j(W),f_\calP^{(j)}).
\]
Theorem \ref{thm:fund} is adapted in a straightforward way to give formulas for the values of the Green functions $\Phi^{(j)}(z,h,f_\calP^{(j)})$ at the cycles $Z_j(W)$ on $X_{K,j}$. We essentially have to exchange the roles of the indices $1$ and $j$.
In that way we obtain
\begin{align*}
\tilde\Phi^{(j)}(Z_j(W),f_\calP^{(j)})
&=-
\frac{\deg(Z_j(W))}{\sqrt{D}}\cdot
\int_{\tilde \Gamma_\infty\bs \H^d}^{reg}
\big\langle \overline{\delta ( f_\calP^{(j)})}, \,\Theta_P\otimes \calE_\N^{(j)}-  2\calE_\LL^{(j)}\big\rangle v^\kappa d\mu(\tau).
\end{align*}
Summing over $j$  we find
\begin{align*}
\tilde\Phi(Z(W),\calP)&=-
\frac{\deg(Z(W))}{\sqrt{D}}\int_{\tilde \Gamma_\infty\bs \H^d}^{reg}
\big\langle \overline{\delta ( f^{(1)}_\calP)}, \,\Theta_P\otimes \calE_\N-  2\calE_\LL\big\rangle v^\kappa d\mu(\tau)\\
&\phantom{=}{}+
\frac{\deg(Z(W))}{\sqrt{D}}\int_{\tilde \Gamma_\infty\bs \H^d}^{reg}
\big\langle \overline{\delta ( f^{(1)}_\calP)}, \,\Theta_P\otimes \calE_\N^{(0)}-  2\calE_\LL^{(0)}\big\rangle v^\kappa d\mu(\tau)
.
\end{align*}
Notice that $\delta ( f^{(1)}_\calP)=\delta ( f^{(j)}_\calP)$ for all $j$.
In the first summand, the function $\Theta_P\otimes \calE_\N-  2\calE_\LL$ is modular. By means of the unfolding argument we see that the first summand is equal to
\begin{align*}
&-
\deg(Z(W))\int_{\tilde \Gamma\bs \H^d}^{reg}
\big\langle \overline{\xi ( f^{(1)}_\calP)}, \,\Theta_P\otimes \calE_\N-  2\calE_\LL\big\rangle v^\kappa d\mu(\tau)\\
&=- \deg( Z(W))\left( \Theta_P\otimes \calE_\N, \,\xi ( \calP)\right)_{Pet}
+2 \deg( Z(W)) \left( \calE_\LL,\,\xi ( \calP)\right)_{Pet}\\
&=- \deg( Z(W))L'(\xi ( \calP), W,0).
\end{align*}
In the last equality we have used that the Petersson scalar product of a cusp form and an Eisenstein series vanishes.

We now compute the quantity
\begin{align}
\label{eq:q2tr}
\frac{1}{\sqrt{D}}\int_{\tilde \Gamma_\infty\bs \H^d}^{reg}
\big\langle \overline{\delta ( f_\calP^{(1)})}, \,\Theta_P\otimes \calE_\N^{(0)}-  2\calE_\LL^{(0)}\big\rangle v^\kappa d\mu(\tau).
\end{align}
Here we use that $\Theta_P\otimes \calE_\N^{(0)}-  2\calE_\LL^{(0)}$ is holomorphic on $\H^d$. As in the proof of Theorem \ref{thm:fund}, we obtain by Stoke's theorem that \eqref{eq:q2tr} is equal to
\begin{align*}
&- \frac{1}{\sqrt{D}}\int_{\tilde \Gamma_\infty\bs \H^d}^{reg}
d \left\langle f_\calP^{(1)}, \Theta_P\otimes \calE_\N^{(0)}\eta
-  2\calE_\LL^{(0)}\eta
\right\rangle\\
&=\lim_{v_1\to \infty} \int_{v_2,\dots,v_d=0}^\infty
g_0(v)(v_2\cdots v_d)^{\ell/2-2}\, dv_2\dots dv_d\\
&\phantom{=}
-{}\lim_{v_1\to 0} \int_{v_2,\dots,v_d=0}^\infty
g_0(v)(v_2\cdots v_d)^{\ell/2-2}\, dv_2\dots dv_d,
\end{align*}
where $g_0(v)$ denotes the constant term of the Fourier series
\[
g(\tau)=\left\langle f_\calP^{(1)}, \Theta_P\otimes \calE_\N^{(0)}
-  2\calE_\LL^{(0)}
\right\rangle.
\]
Since $g_0(v)=O(v_1^{n/2})$, as $v_i\to 0$, the limit $v_1\to 0$ vanishes.
The limit $v_1\to \infty$ is equal to
\[
\CT\left\langle \calP, \Theta_P\otimes \calE_\N^{(0)}
-  2\calE_\LL^{(0)}
\right\rangle.
\]
This concludes the proof of the theorem.
\end{proof}

We now state the main result of \cite{Br2} in a form which is convenient for the present paper.

\begin{theorem}
\label{thm:bop}
Let $\calP$ be a principal part polynomial as in \eqref{eq:pppol}
with integral coefficients $c(m,\mu)\in \Z$.
Assume that $\calP$ is weakly holomorphic, that is, $\xi(\calP)=0$.
Then there exists a function  $\Psi^{(j)}(z,h,\calP)$ on $\D \times H_j(\hat \Q)$ with the following properties:
\begin{enumerate}
\item[(i)]
$\Psi^{(j)}$ is a meromorphic modular form for $H_j(\Q)$ of level $K$, with a unitary multiplier system of finite order, of weight $-B(\calP)/2$, where
\[
B(\calP)=\sum_{\mu\in L'/L} \sum_{m\gg 0} c(m,\mu)B_\LL(m,\mu).
\]
Here $B_\LL(m,\mu)$ is the $(m,\mu)$-th coefficient of the Eisenstein series $E_\LL(\tau,s_0,\kappa)$.
\item[(ii)]
The divisor of $\Psi^{(j)}$ is equal to $\frac{1}{2}Z_j(\calP)$.
\item[(iii)]
The Petersson metric of $\Psi^{(j)}$ (normalized as in \cite{Br2}) is given by
\[
-2\log\|\Psi^{(j)}(z,h,\calP)\|_{Pet}^2 =
\tilde\Phi^{(j)}(z,h,f_\calP^{(j)})+2\CT\langle \calP,\calE^{(0)}_\LL\rangle.
\]
\end{enumerate}
\end{theorem}

Note that $\Psi^{(j)}(z,h,\calP)$ is uniquely determined up to a locally constant multiple of modulus $1$.

Let  $\calP$ be a weakly holomorphic principal part polynomial as in \eqref{eq:pppol}
with integral coefficients  and assume that $B(\calP)=0$. Then there is a positive integer $r$ such that
$\Psi^{(j)}(z,h,r\calP)$ is a rational function on $X_{K,j}$ for $j=1,\dots,d$.
Replacing $\calP$ by $r\calP$, we may assume without loss of generality that $r=1$.
We write $\Psi(z,h,\calP)$ for the rational function on
$\X_K(\C)$
whose restriction to $X_{K,j}$ is equal to
$\Psi^{(j)}(z,h,\calP)$.
The following corollary generalizes the main result of \cite{Scho} to quadratic spaces over totally real fields.

\begin{corollary}
\label{cor:s1}
Let $\calP$ be as above and assume that  $Z(W)$ and $Z(\calP)$ do not intersect on $\X_K(\C)$.  Then the value of $\Psi(z,h,\calP)$ at the CM cycle $Z(W)$ is given by
\begin{align*}
\log|\Psi(Z(W),\calP)| &=\sum_{j=1}^d\log|\Psi^{(j)}(Z_j(W),\calP)|\\
&=-\frac{\deg(Z(W))}{4}\cdot \CT\langle \calP, \Theta_P\otimes \calE_\N^{(0)}\rangle.
\end{align*}

Let $S(\N)$
be the set of finite primes $\mathfrak p$ of $F$ for which  $\N_{\mathfrak p}$ is not
  unimodular (with respect to $\psi_{\mathfrak p}$), and let $S(\calP)$ be the set of totally positive numbers $m\in F$ such that
  $c(m, \mu)\ne 0$ for some $\mu \in L'/L$.
We have
\begin{align*}
\log|\Psi(Z(W),\calP)| &=
\sum_{\text{$p$ prime}} \alpha_p \log(p)
\end{align*}
with coefficients $\alpha_p\in \Q$, and $\alpha_p=0$ unless  there is a prime $\mathfrak p$  of $F$ above $p$  which belongs to $S(\N)$ or
$\mathfrak p| (m-Q(\nu))\partial$ for some $m \in S(\calP)$ and $\nu \in P'$ with $m -Q(\nu)\gg 0$.
In particular, $\alpha_p=0$ unless $p \le \max( M(\calP),  |\N'/\N|, D)$, where
$$
  M(\calP)=\max\{\norm(m)D;\; m\in S(\calP)\}.
$$
\end{corollary}

\begin{proof}
The first part is a direct consequence of Theorem \ref{thm:fundtr}
and Theorem \ref{thm:bop}. By definition, one has
$$
\CT\langle \calP, \Theta_P\otimes \calE_\N^{(0)}\rangle
 = -\frac{2^d}{\Lambda(1,\chi)}\sum_{\substack{ m\in S(\calP)\\ \mu \in L'/L}} c(m, \mu)  \sum_{\substack{\nu \in P'/P,
 \nu'\in \N'/\N, \\\mu=\nu + \nu' \\ m=n + n'}}r_P(n, \nu)
 \beta^*_\N(n', \nu'),
$$
where $r_P(n, \nu)$ is the number of $\nu_1 \in \nu +P$ with
$Q(\nu_1) =n$ (so $n=0$ or $n \gg 0$ is totally positive), and
$\beta^*_\N(n', \nu')$ is the $(n', \nu')$-th coefficient of
$\mathcal E_\N^{(0)}(\tau)$ defined in Corollary  \ref{cor4.10} (so
$n'=0$ or $n' \gg 0$). Since $Z(W)$ and $Z(\calP)$ do not intersect
in $\mathbb X_K(\mathbb C)$, we have $n' =m-n \gg 0$ when $c(m, \mu) \ne 0$. In
such a case, Corollary \ref{cor4.10} implies that
$$
\beta^*_\N(n', \nu') =a_p \log p
$$
for some rational number $a_p \in \Q$. Moreover, $a_p=0$ unless
$\Diff(\N, n') =\{ \mathfrak p\}$ for some prime of $F$ above $p$,
and $\mathfrak p \in S(\N)$ or $\ord_\mathfrak p (m-n) \partial \ge
0$ is odd. Now the second part follows.
\end{proof}

Since $\X_K$ and $Z(\calP)$ are defined over $F$, there exists a rational function $R_\calP$ on $\X_K$ defined over $F$
such that the corresponding functions $\sigma_j(R_\calP)$ on $X_{K,j}$ satisfy
\[
\sigma_j(R_\calP)(z,h)= C_j(z,h) \Psi^{(j)}(z,h,\calP),
\]
where $C_j:X_{K,j} \to \C$ is a locally constant function. We let $C:\X_K(\C)\to\C$ be the locally constant function whose restriction to $X_{K,j}$ is equal to
$C_j$. It is an interesting question whether one can choose $R_\calP$ such that all values of $C$  have modulus $1$. Then $C$ could be absorbed in the normalizing constants of the functions $\Psi^{(j)}(z,h,\calP)$.

The CM value
\[
R_\calP(Z(W)) := \prod_{a\in  Z(W)} R_\calP(a):= \prod_{a \in \supp(Z(W))} R_\calP(a)^{\frac{2}{w_K}}
\]
lies in $F$.  The following corollary describes the prime factorization of the norm of this quantity.

\begin{corollary}
\label{cor:s2}
Under the above assumptions we have
\begin{align*}
\norm_{F/\Q} R_\calP(Z(W)) &= \pm |C(Z(W))| \cdot\exp\left(-\frac{\deg(Z(W))}{4}  \CT\langle \calP, \Theta_P\otimes \calE_\N^{(0)}\rangle\right),
\end{align*}
where $C(Z(W)) =\prod_j C_j(Z_j(W))$.
\end{corollary}

\begin{proof}
The norm of $R_\calP(Z(W))$ is given by
\begin{align*}
\norm_{F/\Q} R_\calP(Z(W)) &= \prod_{j=1}^d \sigma_j(R_\calP(Z(W))) \\
&= C(Z(W))\cdot \Psi(Z(W),\calP).
\end{align*}
Hence the statement follows from Corollary \ref{cor:s1}.
\end{proof}


\begin{remark} \label{rem7.6}
The map
$$
Z_j(W) \rightarrow X_{K, j} \rightarrow \pi_0(X_{K, j})=F_+^\times \backslash \hat{F}^\times/\nu(K)
$$
is surjective when $\hat{F}^\times =\nu(K)\norm(\hat{\kay}^\times) F^\times$, where   $\nu$ is the spinor norm on $H_j$ and $F_+^\times$ is the subgroup of $F^\times$ of  the totally positive elements.
This is for instance the case if there is a prime $\frakp$ of $F$ ramified in $\kay$ such that $\calO_{F,\frakp}^\times \subset \nu(K)$.
In such a case
$$
C(Z(W)) = \norm(C)^{\frac{\deg Z(W)}r},
$$
where $\norm(C)$ denotes the product of the values of $C$ over the connected components of $\X_K(\C)$, and
$r =|\pi_0(X_{K, j})|$ does not depend on $j$.
\end{remark}

\begin{remark}
\label{rem:integrals}
Note that for $d=1$ the formulas of Theorem \ref{thm:fundtr} and Corollary \ref{cor:s1} are compatible with \cite{Scho}, \cite{BY} and \cite{Ku:Integrals}.
The Green function
$\tilde\Phi(z,h,\calP)$ in the present paper is equal to the Green function in the other papers only up to an additive constant which can be fixed by specifying $\int_{X_{K,1}}\tilde\Phi(z,h,\calP)\Omega^n$.
The Green function $\tilde\Phi(z,h,\calP)$ has vanishing integral over $X_{K,1}$ and in the formula for the CM values the extra term $-\deg(Z(W))\alpha(\calP)$ occurs, where $\alpha(\calP)= 2\CT\langle \calP,\calE_\LL^{(0)}\rangle$.
Hence the Green function $\tilde\Phi(z,h,\calP)+\alpha(\calP)$ has integral $\vol(X_{K,1})\alpha(\calP)$ over $X_{K,1}$ and no extra term in the CM value formula. This Green function has the same additive normalization as the ones in \cite{Scho}, \cite{BY} and \cite{Ku:Integrals}.
\end{remark}

\section{Examples}
\label{sect:8}

As an example we consider the Shimura curve $X$ associated to the triangle group $G_{2,3,7}$, see \cite{El} Section 5.3 and \cite{El2} Section 2.3. It is a genus zero curve with a number of striking properties. For instance, the minimal quotient area of a discrete subgroup of $\operatorname{PSL}_2(\R)$ is $1/42$, and it is only attained by the triangle group $G_{2,3,7}$.
Elkies constructed a generator $t$ of the function field of $X$ and computed its values at certain CM points. Here we show that this function is a regularized theta lift in the sense of Theorem \ref{thm:bop}. Employing Corollary \ref{cor:s2}, we verify some of Elkies' computations and determine some further CM values of $t$. The results are summarized in Table \ref{tab:1} below\footnote{A Magma program for the explicit evaluation of the formula of Corollary \ref{cor:s2} can be obtained from the authors.}.

Let $F =\mathbb Q(\zeta_7)^+$ be the maximal totally real subfield of the cyclotomic field $\mathbb Q(\zeta_7)$, where $\zeta_7=e^{2\pi i/7}$. Then $F$ is a Galois extension of $\Q$ which is generated by $\alpha=\zeta_7+\zeta_7^{-1}$. The minimal polynomial of $\alpha$ is $x^3+x^2-2x-1$, and the ring of integers $\calO_F$ of $F$ is given by $\Z[\alpha]$. The field $F$ has narrow class number $1$ and discriminant $49$. The prime ideal above the totally ramified prime $7$ is generated by
$(1-\zeta_7)(1-\zeta_7^{-1}) = 2 -\alpha$. The different has the totally positive generator $\delta=(2-\alpha)^2$.

Let $\sigma_1, \sigma_2, \sigma_3$ be the three real embeddings of
$F$. Let $B$ be the (up to isomorphism unique) quaternion algebra over
$F$ which is ramified exactly at the two infinite places $\sigma_2$ and
$\sigma_3$. Let $\OO_B$ be a fixed maximal order of $B$. We will
identify $\hat{B}$ with $M_2(\hat{F})$ in such a way that
$\hat{\OO}_B$ is identified with $\Mat_2(\hat{\OO}_F)$.  We consider
the quadratic space
$$
V =\{ x \in   B;  \;  \tr x =0 \},  \quad Q(x) = \delta^{-1}
\det x = -\delta^{-1} x^2,
$$
where $\det x$ is the reduced norm. The $\calO_F$-lattice $L = V \cap \OO_B$ is even and integral.
We have
$$\hat{L} = \{ x =\kzxz {b} {a} {c} {-b};  \;  a , b, c \in \hat{\OO}_F\},
$$
and $L'/L \cong \OO_F/2\OO_F\cong \F_8$, since the prime $2$ is inert in $F$. In this case, $ H =\Gspin(V)$ can be
identified with $\Res_{F/\Q} B^\times$, which acts on $V$ by
conjugation.  The compact open subgroup
$K=\hat{\OO}_B^\times=\GL_2(\hat{\OO}_F) \subset H(\hat \Q)$ preserves
$\hat{L}$ and acts trivially on $L'/L$. The associated Shimura curve
$X =X_{K, 1}$ over $\C$ is given by
$$
X= B^\times \backslash \mathbb H^\pm \times \widehat{
B}^\times/K  = \OO_{B}^1 \backslash \mathbb H
$$
by the strong approximation theorem.  Here $\OO_B^1$ is the group of
norm one elements in $\OO_B$ and $\mathbb H$ denotes the usual upper
complex half plane.  The Shimura curve $X$ has a canonical model over
$F$, and its Galois conjugates are $X_{K, 2}$ and $X_{K, 3}$ as
discussed in Section \ref{sect:global}.  It can be shown that $X$ has actually a
model over $\Q$, see \cite[Section 5.3]{El}, and that the curves $X_{K, i}$
are isomorphic to each other. We remark that in the setting of Section
\ref{sect:global}, the Shimura curve $X$ is associated to the
incoherent quaternion algebra $\B$ which is split at all finite places
and ramified at all infinite places.

\subsection{CM cycles}

For a totally imaginary quadratic field extension $\kay$ of $F$, we write
${\OO}_\kay =\OO_F + {\OO}_F \gamma$ for some $\gamma\in \calO_k$, and
define an embedding $\hat{\iota}: \hat{\kay} \to
\hat{B}=\Mat_2(\hat{F})$ by
\begin{equation}
\begin{pmatrix}r \\ r\gamma \end{pmatrix} = \hat\iota(r) \begin{pmatrix} 1 \\
\gamma\end{pmatrix}, \quad \text{for $r \in \hat{\kay}$.}
\end{equation}
Then  it is easy to see that $\hat{\iota}^{-1}(\hat{\OO}_B) =\hat{\OO}_\kay$.
We choose and fix an embedding $\iota_\infty:  \kay_\infty
\to B_\infty$. This gives an embedding
$\iota=\hat{\iota} \iota_\infty: \A_\kay \to B(\A_F)$, and induces an embedding $\iota: \kay \to B$.
 To simplify the notation, we drop
the embedding $\iota$ and view $\hat{\OO}_\kay$ as a subring of
$\hat{\OO}_B$, and write
\begin{equation}
B(\A_F) = \A_\kay  \oplus  \A_\kay \xi_\A,
\end{equation}
for some $\xi_\A =\hat{\xi} \xi_\infty \in B(\A_F)$ such that
$\xi_\A^2  \in  \A_F^\times$,  $\hat{\xi}= \kzxz {1} {0} {0} {-1}$, and $\xi_\A r = \bar r \xi_\A$
for  any $r \in \A_\kay$. We consider the totally positive definite subspace
$$
W=V \cap \kay =\{z \in \kay; \;  \tr_{\kay/F}=0\}, \quad  Q(z) = \delta^{-1} z
\bar z =-\delta^{-1} z^2.
$$
We write $k=F(\sqrt\Delta)$
for some square-free totally negative element $\Delta \in \OO_F$.
We obtain the sublattices
\begin{equation}
P =W \cap L =\OO_F \sqrt{\Delta}, \quad N =W^\perp \cap L.
\end{equation}
Then we have
\begin{equation}
(P, Q) \cong (\OO_F, -\frac{\Delta}{\delta} x^2),  \quad  P' \cong
\frac{1}{2\Delta} \OO_F,
\end{equation}
and
\begin{equation} \label{eq:N}
(\hat{N}, Q) =(\hat{\OO}_\kay \hat{\xi}, \delta^{-1} \det ) \cong (\hat{\OO}_\kay, -\delta^{-1} r \bar r), \quad
\hat{N}' \cong \hat{\partial}_{\kay/F}^{-1}.
\end{equation}
In this section, we denote the CM cycle $Z(W)$ corresponding to $W$
by
$$
Z(\OO_\kay) =\kay^\times \backslash \{ z_W^\pm \} \times \hat{\kay}^\times/\hat{\OO}_\kay^\times,
$$
and identify it with its image in $X$.  Notice that $z_W^\pm$ collapse
to one point in $X$. So $Z(\OO_\kay)$ has $ h_\kay$ points and each
point is counted with multiplicity $\frac{4}{w_\kay}$. When $\kay =
F(\sqrt{-4})$, $F(\sqrt{-3})$, $F(\sqrt{-7})=\Q(\zeta_7)$,
$F(\sqrt{-8})$, $F(\sqrt{-11})$, we have $h_\kay =1$. We denote the corresponding
unique point in $Z(\OO_\kay)$ by $P_4$, $P_3$, $P_7$, $P_8$, and
$P_{11}$ respectively (counted with multiplicity $1$).
Note that the point $P_4$ is denoted by $P_2$
in \cite{El}.  According to \cite{El}, these points are all defined
over $\Q$.

\subsection{The divisors $Z(m, \mu)$}

The discriminant group $L'/L = \frac{1}2 \OO_F/\OO_F
\cong \mathbb F_8$ is a field, and the $\F_8$-valued quadratic form on $L'/L$ induced by $Q$ is given by the Frobenius automorphism.
Consequently, if $m \in \frac{1}{4\delta} \OO_F$ is totally positive, then there exists at most one $\mu \in L'/L$ such that
$m \in Q(\mu) +\partial^{-1}$.
So we will simply
write $Z(m) =Z(m, \mu)$.
By a similar argument as in \cite[Lemma 7.2]{BY}, one sees that if
$m = -\frac{d}{4\delta}$ such that $\kay_d=F(\sqrt{d})$ is a CM
number field with relative discriminant $d_{\kay/F} = d\OO_F$, one has
$$
Z(m) = Z(\OO_{d})
$$
as divisors on $X$. Here we have briefly written $\calO_d$ for the ring of integers $\calO_{k_d}\subset k_d$.

\subsection{Elkies' rational  function}

\label{sect:8.3}

By the above discussion, one sees that $Z(\frac{4}{4\delta})$ is Elkies' CM point $P_4$ with multiplicity $1$, and
$Z(\frac{d_7}{4 \delta})$ with $d_7 = 2-\alpha$ is Elkies' CM point $P_7$ with multiplicity $\frac{2}{7}$.
We consider the principal part polynomial
$$
\mathcal P=2q^{-\frac{4}{4\delta}} \chi_{0} - 7 q^{-\frac{d_7}{4\delta}} \chi_{\mu_7},
$$
where $\mu_7\in L'/L$ is the unique coset such that $\frac{d_7}{4\delta}-Q(\mu)\in \partial^{-1}$. The corresponding divisor is given by
$$
Z(\mathcal P)=  2Z(\frac{4}{4\delta}) - 7 Z(\frac{d_7}{4 \delta})= 2P_4  -2 P_7.
$$
We aim to lift $\calP$ to a rational function on $X$ by means of Theorem \ref{thm:bop}.

First, we need to know that $\calP$ is weakly holomorphic, i.e., that $\xi(\calP)\in S_{\kappa,\rho_L}$ vanishes.
In fact, we have that $S_{\kappa,\rho_L}=\{0\}$, which can be seen as follows: If $g=\sum_{\mu}g_\mu \chi_\mu\in S_{\kappa,\rho_L}$, then $\tilde g(\tau):=\sum_{\mu}g_\mu(4\tau) $ is a scalar valued Hilbert cusp form of weight $\kappa$ for the group $\Gamma_0(4)$ in the sense of Shimura. Its Shimura lift $S(\tilde g)$ is a scalar valued Hilbert cusp form of weight $2\kappa-1$ for the group $\Gamma_0(2)$, see  \cite{Sh2}. Moreover, the maps $g\mapsto \tilde g\mapsto S(\tilde g)$ are injective.
A dimension computation shows that $S_2(\Gamma_0(2))$ vanishes. This proves the claim.

Second, we have to compute the weight of the lift $\Psi(z,h,\calP)$.
According to \eqref{eq:divf2}, it is given by
\begin{align*}
-\frac{B(\calP)}{2} &= -B_L\left(\frac{4}{4\delta},0\right)
+ \frac{7}{2} B_L\left(\frac{d_7}{4 \delta},\mu_7\right)\\
&= \frac{\deg(Z(\frac{4}{4\delta},0))}{\vol(X_K)}
- \frac{7}{2} \frac{\deg(Z(\frac{d_7}{4 \delta},\mu_7))}{\vol(X_K)}\\
&= 0.
\end{align*}
Hence, the lift $\Psi(z,h,\calP)$ of $\calP$ is a rational function on $X$ with a double zero at $P_4$ and a septuple pole at $P_7$. It must agree with Elkies' function $t$ (see \cite[Section 5.3]{El}) up to a constant multiple.
Since Elkies' function
is defined over $\Q$, we can take it for the function $R_\mathcal P$ in Corollary \ref{cor:s2}. We obtain in that way:

\begin{proposition}
Let $t$ be Elkies' rational function on $X$ with a double zero at $P_4$, a pole of order $7$ at $P_7$, and $t(P_3)=1$.
Put $\tilde t = 2^6\cdot 3^3 \cdot t$. Then for any CM cycle $Z(W)$ which is disjoint from $\dv(\tilde t)$, we have
\begin{align} \label{eq:CMvalue}
\norm_{F/\Q}\big(\tilde t(Z(W))\big)
&=\exp\left( -\frac{\deg Z(W)}{4} \CT\langle \mathcal P, \Theta_P\otimes \mathcal E_N^{(0)}\rangle\right)\\
\nonumber
&=\prod_{m\in \{\frac{4}{4\delta}, \frac{d_7}{4\delta}\}}\prod_{\substack{x\in P'\\ Q(x)\ll m}} \exp\big( \beta_N^*(m-Q(x),x)\big)^{c(m)}.
\end{align}
Here $c(\frac{4}{4\delta})=2$, $c(\frac{d_7}{4\delta})=-7$, and $ \beta_N^*(m,\mu)$ denote the coefficients of $\calE_N^{(0)}$ given in Corollary~\ref{cor4.10} and Proposition~\ref{prop:ex}.
\end{proposition}

\begin{proof}
Since the $X_{K, i}$ are connected, Corollary \ref{cor:s2} and Remark \ref{rem7.6} assert that
 there is a positive constant $C$, independent of $W$, such that
$$
\norm_{F/\Q} \tilde t(Z(W))=\pm C^{\deg Z(W)}  \exp\left( -\frac{\deg Z(W)}{4} \CT\langle \mathcal P, \Theta_P\otimes \mathcal E_N^{(0)}\rangle\right).
$$
The constant $C$ can be determined by evaluating at
$Z(W) =Z(\OO_{-3})=\frac{2}{3} P_3$. The norm of the value of $\tilde t$ at this cycle is $2^{12}\cdot 3^{6}$. On the other hand, evaluating our formula, we obtain
$|\Psi(Z(\OO_{-3}),\calP)|=2^{12}\cdot 3^{6}$, so that $C=1$. This implies the first formula.

The second formula follows by inserting the Fourier expansion of $\calE_N^{(0)}$ (Corollary \ref{cor4.10}) and the formula $\deg Z(\OO_\kay) =\frac{4}{w_\kay} h_\kay$ (see \eqref{eq:degCM2}) in the exponential. Note that using the class number formula
$$
\Lambda(1, \chi_{\kay/F}) = \frac{2}{w_\kay} \frac{h_\kay}{h_F} \frac{R_\kay}{R_F} = \frac{2}{w_\kay} h_\kay \frac{R_\kay}{R_F},
$$
and the fact that $\frac{R_k}{R_F}=4$, we see that
\[
\frac{\deg Z(W)}{4}\cdot \frac{2^3}{\Lambda(1,\chi)} = \frac{h_\kay}{w_\kay}\cdot  \frac{2^3 w_k R_F}{2 h_k R_k}=\frac{4R_F}{R_k} =1.
\]
\end{proof}

Table \ref{tab:1} contains the values of $\tilde t $ at CM cycles $\frac{1}{2}Z(\calO_{d})$ for a few (mainly) odd  relative discriminants $d\in \calO_F$.
The values for $d=-8$ and $d=-11$ were previously computed by Elkies.
We evaluated the formulas of Corollary \ref{cor:s2} and Proposition~\ref{prop:ex} by means of a Magma program. Since we use Proposition~\ref{prop:ex}, we have to assume that $d$ is coprime to $2$ (otherwise we can only compute the CM value up to a power of $2$). There are exactly 6 rational primes $p\leq 1000$ that split in $\calO_F$ and for which there exists a totally negative prime element $d\in \calO_F$ such that $d\calO_F\cap\Z= (p)$  and such that $k_d/F$ is unramified at $2$. These are the primes $167$, $239$, $251$, $379$, $491$, $547$. The CM values corresponding to the first three primes of this list are given at the end of the table. Recall that $\alpha = \zeta_7 +\zeta_7^{-1}$ in the table.

\begin{table}[h]
\begin{center}
\caption{\label{tab:1} Values of $\tilde t$ at the CM cycles $\tfrac{1}{2}Z(\calO_d)$}
\begin{tabular}{|r|c|c|c| }
  \hline \rule[-3mm]{0mm}{8mm}
  $d$ & $h(k_d)$ & $2/w_{k_d}$ & $\norm_{F/\Q} \tilde t\left(\tfrac{1}{2}Z(\calO_d)\right)$   \\
  \hline \hline\rule[-3mm]{0mm}{8mm}
  $-3$ & $1$ & $1/3$ &   $2^{6}\cdot 3^{3}$  \\
  \hline \rule[-3mm]{0mm}{8mm}
  $-4$ & $1$ & $1/2$   &   $0$\\
  \hline \rule[-3mm]{0mm}{8mm}
  $\alpha-2$& $1$ & $1/7$&$\infty$\\
  \hline  \rule[-3mm]{0mm}{8mm}
  $-8$  & $1$ & $1$  &  $\frac{2^*\cdot 7^{9}\cdot 167^{6}\cdot 239^{6}}
  {13^{21}}$\\
  \hline \rule[-3mm]{0mm}{8mm}
  $-11$ & $1$ & $1$  &  $\frac{2^{18}\cdot 7^{9}\cdot  11^{3}\cdot 43^{6}\cdot 127^{6} \cdot 139^{6}\cdot 307^{6}\cdot
  659^{6}}{13^{21}\cdot 83^{21}}$\\
  \hline \rule[-3mm]{0mm}{8mm}
  $-15$ & $2$ & $1$ & $\frac{ 3^{18}\cdot 7^{18}\cdot 11^{6}\cdot 43^{12}\cdot 71^{12}\cdot 839^{6}\cdot 911^{6}\cdot 2099^{6}\cdot 2339^{6}}
  {13^{42}\cdot 41^{21}\cdot 251^{15}}$\\
  \hline \rule[-3mm]{0mm}{8mm}
  $-19$ & $3$ & $1$ & $\frac{ 2^{54}\cdot 19^{3}\cdot 71^{6}\cdot 127^{6}\cdot 211^{6}\cdot 223^{6}
  \cdot 743^{6}\cdot 911^{6}\cdot 1091^{6}\cdot 1399^{6}\cdot 2339^{6}\cdot 2659^{6}\cdot 2687^{6}
  \cdot 3571^{6}\cdot 4787^{6}\cdot 5167^{6}}
  {3^{24}\cdot 13^{63} \cdot 41^{21}\cdot 167^{21}\cdot 307^{21}}$\\
  \hline \rule[-3mm]{0mm}{8mm}
  $-23$& $9$ & $1$ & $\frac{7^{69}\cdot 11^{12}\cdot 19^{6}\cdot 23^{3}\cdot 43^{54}\cdot 251^{12}\cdot
  503^{12}\cdot 743^{12}\cdot 911^{6}\cdot 1091^{6}\cdot 5167^{6}\cdot 5839^{6}}
  {83^{18}\cdot 97^{21}\cdot 181^{21}\cdot 419^{21}\cdot 1049^{21}}$\\
  \hline \rule[-3mm]{0mm}{8mm}
  $1-8\alpha^2$ & $1$ & $1$ & $\frac{ 43^{6}\cdot 71^{2}\cdot 83^{2}\cdot 167}
  {13^{7}}$\\
  \hline \rule[-3mm]{0mm}{8mm}
  $\alpha^2-8$ & $1$ & $1$ & $\frac{7^{7}\cdot 43^{2}\cdot 71^{2}\cdot 139^{2}\cdot 239}
  {13^{7}}$\\
  \hline \rule[-3mm]{0mm}{8mm}
  $4\alpha-7$ & $1$ & $1$ & $\frac{2^{18}\cdot 43^{2}\cdot 71^{2}\cdot 83^{2}\cdot 127^{4}\cdot 251}
  {13^{14}}$\\
  \hline
\end{tabular}
\end{center}
\end{table}

Some further examples can be constructed in this case as follows. If $k/F$ is a CM extension as above, and $d$ is a totally negative generator of the discriminant of $k/F$, then
\[
\tilde t_d(z):= \prod_{a\in Z(\calO_d)}\big(\tilde t(z)-\tilde t(a)\big)
\]
is a rational function on $X$ with divisor $Z(\calO_d)-\deg(Z(\calO_d))\cdot P_7$. Up to a constant multiple, it is the regularized theta lift $\Psi(z,\calP)$ of the principal part polynomial
$\calP = q^{d/4\delta}\chi_{\mu_d}-7\deg(Z(\calO_k)) q^{-d_7/4\delta}\chi_{\mu_7}$.  Therefore the CM values of this function can be determined using Corollary
\ref{cor:s2}. Note that $\tilde t_{-4}= \tilde t$.

The values of $\tilde t_{-8}(z) = \tilde t(z)-\tilde t(P_8)$ at the CM points in Elkies' list can be computed using his data. On the other hand, for the odd discriminants $d=-3, -11$ we computed these values with our formula and found that the results agree.

Finally, we remark that it would be interesting to study the examples in \cite[Section 7]{Vo} in a similar way.

\end{document}